\documentclass{amsart}
\usepackage{enumerate, amssymb, stmaryrd, amsmath}
\usepackage[colorlinks=true,pagebackref]{hyperref}

\usepackage[all]{xy}
\SelectTips{eu}{} 
\xyoption{curve}

\def\lto{{\longrightarrow}}

\def\xto{\xrightarrow}

\newcommand{\fc}{{\mathfrak c}}

\newcommand{\cala}{{\mathcal A}}

\newcommand{\calk}{{\mathcal K}}

\newcommand{\calt}{{\mathcal T}}

\newcommand{\KK}{{\mathbb K}}
\newcommand{\LL}{{\mathbb L}}

\newcommand{\NN}{{\mathbb N}}

\renewcommand{\SS}{{\mathbb S}}

\newcommand{\ZZ}{{\mathbb Z}}

\newcommand{\bfa}{{\mathbf a}}
\newcommand{\be}{\ensuremath{\mathbf e}}
\newcommand{\bff}{{\mathbf f}}

\newcommand{\bx}{{\mathbf x}}

\newcommand{\bdot}{\bullet}

\DeclareMathOperator{\car}{char}

\DeclareMathOperator{\Cliff}{Cliff}

\newcommand{\diff}{{\mathbf d}}

\DeclareMathOperator{\Der}{Der}

\DeclareMathOperator{\End}{ End}

\DeclareMathOperator{\Ext}{Ext}
\DeclareMathOperator{\ev}{ev}

\DeclareMathOperator{\Hoch}{HH}

\DeclareMathOperator{\Hom}{Hom}
\DeclareMathOperator{\id}{id}

\DeclareMathOperator{\jac}{jac}
\DeclareMathOperator{\Ker}{Ker}

\DeclareMathOperator{\op}{op}

\DeclareMathOperator{\Sym}{{\mathbb S}ym}

\DeclareMathOperator{\tOmega}{\widetilde{\Omega}}
\DeclareMathOperator{\Tor}{Tor}

\newcommand{{\sbullet}}{{\scriptstyle\bullet}}
\newcommand{\gs}{\geqslant}

\theoremstyle{definition}
\newtheorem{defn}{Definition}[section]

\newtheorem{remark}[defn]{Remark}

\newtheorem{sit}[defn]{}
\newtheorem{example}[defn]{Example}
\newtheorem{examples}[defn]{Examples}

\theoremstyle{plain}

\newtheorem{proposition}[defn]{Proposition}
\newtheorem{theorem}[defn]{Theorem}

\newtheorem{lemma}[defn]{Lemma}
\newtheorem{cor}[defn]{Corollary}

\begin{document}
\title[Multiplicative Structure of Hochschild Cohomology]
{The Multiplicative Structure on Hochschild Cohomology of a Complete Intersection}

\dedicatory{Dedicated to Hans-Bj\o rn Foxby in admiration}

\author[R.-O. Buchweitz]{Ragnar-Olaf Buchweitz}
\address{Dept.\ of Computer and Math\-ematical Sciences,
University  of Tor\-onto at Scarborough, 
1265 Military Trail, 
Toronto, ON M1C 1A4,
Canada}
\email{ragnar@utsc.utoronto.ca}

\author[C. Roberts]{Collin Roberts}
\address{Dept.\ of Pure Mathematics, University of Waterloo, 
200 University Avenue West, 
Waterloo, ON  N2L 3G1, 
Canada}
\email{cd2rober@math.uwaterloo.ca}

\thanks{The authors were partly supported by NSERC grant 3-642-114-80 and 
this material is based upon work supported by the National Science Foundation under Grant No. 
0932078 000, while the first author was in residence at the Mathematical Science Research 
Institute (MSRI) in Berkeley, California, during the semester of Jan.--May 2013. 
}
\date{\today}

\begin{abstract}
We determine the product structure on Hochschild cohomology of commutative algebras in low degrees, 
obtaining the answer in all degrees for complete intersection algebras. As applications, we consider 
cyclic extension algebras as well as Hochschild and ordinary cohomology of finite abelian groups.
\end{abstract}


\maketitle 
\setcounter{tocdepth}{1}
{\footnotesize\tableofcontents}
\section{Introduction}

For any associative algebra%
\footnote{All our rings have a multiplicative identity and ring homomorphisms are unital.}
$A$ over a commutative ring $K$, the Hochschild cohomology ring $\Hoch(A/K,A)$ with its
cup (or Yoneda) product is {\em graded commutative\/} by Gerstenhaber's fundamental result \cite{Ger1}. 

However, it is not necessarily {\em strictly\/} graded commutative, that is, the squaring operation 
$\Hoch^{2i+1}(A/K,A)\to  {}_{2}\Hoch^{4i+2}(A/K,A)$, for $i\geqslant 0$, from Hochschild cohomology 
in odd degrees  to, necessarily, the $2$--torsion of the Hochschild cohomology in twice that degree, 
might be nontrivial.

Indeed, this is already so in the simplest example. If $A=K[x]/(x^{2})$ with $2=0$ in $K$, then the 
Hochschild cohomology ring of $A$ over $K$ is isomorphic to a polynomial ring  over $A$ in one variable, 
$\Hoch(A/K,A)\cong A[z]$, where $A\cong \Hoch^{0}(A/K,A)$ is of cohomological degree zero, 
and the variable 
$z$ is of degree one, representing the class $\partial/\partial x\in \Der_{K}(A,A)\cong \Hoch^{1}(A/K,A)$. 
As we are in characteristic $2$, this ring is indeed graded commutative, but certainly not strictly so, the 
square of $z^{2i+1}\in \Hoch^{2i+1}(A/K,A)$ returning the non-zero element 
$z^{4i+2}\in \Hoch^{4i+2}(A/K,A)$.

The purpose here is to clarify the structure of this squaring map in the simplest case, when $i=0$ and 
$A$ is commutative.
We will then deduce from this the entire ring structure when $A$ is a homological complete intersection over $K$.

The crucial point is that the square of a derivation with respect to cup product, as an element in 
$\Hoch^{2}(A/K,A)$, is determined by the {\em Hessian quadratic forms\/} associated to
the defining equations. 

The key technical tool we use is the {\em Tate model\/} of the multiplication map 
$\mu:A^{\ev}=A\otimes_{K}A\to A$ for
a commutative $K$--algebra $A$. We show that the Postnikov tower of such model carries a 
family of compatible co-unital diagonal approximations
and we give an explicit description of such approximation in low degrees to obtain the desired results.

\subsection*{Acknowledgement}
The present paper relies to a large extent on the second author's thesis \cite{Roberts}.

The authors wish to thank \O yvind Solberg for originally asking how to determine the squaring map on 
Hochschild cohomology of a complete intersection, Luchezar L.~ Avramov and Alex Martsinkovsky for 
valuable help with references, and Lisa Townsley for sharing with them a synopsis of her thesis \cite{Town}. 
We also thank the referee of the paper for careful reading and valuable suggestions regarding presentation 
--- aside from pointing out numerous typos that are hopefully taken care of now.

\section{Statement of Results}
Our treatment requires some simple facts about differential forms and the calculus that goes 
along with them. While we cover all facts we will use, the reader may gain the full picture 
by consulting \cite{EGA}.
\begin{sit}
If $P=K[x\in X]$ is a polynomial ring over some commutative ring $K$, with $X$ a set of variables, then
the module of K\"ahler differentials of $P$ over $K$ is a free $P$--module, based on the differentials $dx$ 
for $x\in X$, that is, $\Omega^{1}_{P/K}\cong \bigoplus_{x\in X}Pdx$.

If now $f\in P$ is any polynomial, then we may consider its {\em Taylor expansion\/} $f(\bx +d\bx)$ in
$\Sym_{P}\Omega^{1}_{P/K}\cong K[x,dx; x\in X]\cong P\otimes_{K}P$.

In concrete terms, the polynomial will only depend on finitely many variables, say, $\bx=x_{1},..., x_{n}$, 
and then 
\[
f(\bx+d\bx)=
\sum_{\bfa=(a_{1},..., a_{n})\in\NN^{n}}\frac{\partial^{(|\bfa|)}f(\bx)}{\partial^{\bfa}\bx}d\bx^{\bfa}\,,
\]
where $|\bfa|=\sum_{i} a_{i}$ and the coefficient of $d\bx^{\bfa}$ 
is the corresponding {\em divided partial derivative\/} of $f$, in that the usual partial derivative satisfies
\begin{align*}
\frac{\partial^{|\bfa|}f(\bx)}{\partial^{a_{1}}x_{1}\cdots \partial^{a_{n}}x_{n}}=
a_{1}!\cdots a_{n}!\frac{\partial^{(|\bfa|)}f(\bx)}{\partial^{a_{1}}x_{1}\cdots \partial^{a_{n}}x_{n}}\,.
\end{align*}
\end{sit}

\begin{sit}
The linear part with respect to the $d\bx$ in that expansion can be viewed as the {\em total differential\/} 
of $f$, to wit, $df=\sum_{i=1}^{n}\frac{\partial f(\bx)}{\partial x_{i}}dx_{i}\in \Omega^{1}_{P/K}$.
 
It defines the $P$--linear {\em Jacobian\/} of the given polynomial,
\begin{align*}
\widetilde{\jac_{f}}&: \Der_{K}(P,P) \lto P\,,\quad
D\mapsto D(df)
\end{align*}
on the $P$--module of $K$--linear derivations from $P$ to $P$.

The quadratic part  with respect to the $d\bx$ in the Taylor expansion,
\[
H_{f}(d\bx) = 
\sum_{1\leqslant i\leqslant j\leqslant n}\frac{\partial^{(2)}f(\bx)}{\partial x_{i}\partial x_{j}}dx_{i}dx_{j}\,,
\]
is the {\em Hessian form\/} defined by $f$. By definition it is an element of
$\Sym_{2}\Omega^{1}_{P/K}$ and so defines a $P$--quadratic form 
\begin{align*}
\tilde h_{f}: \Der_{K}(P,P) \lto P
\end{align*}
that sends $D= \sum_{i=1}^{n}p_{i}\frac{\partial}{\partial x_{i}}+\sum_{x\in X'}p_{x}\frac{\partial}{\partial x}$ 
to 
\[
\tilde h_{f}(D)= (D\otimes D)(H_{f}(d\bx)) =
\sum_{1\leqslant i\leqslant j\leqslant n}\frac{\partial^{(2)}f(\bx)}{\partial x_{i}\partial x_{j}}p_{i}p_{j}\,, 
\]
where we have set $X'=X\setminus\{\bx\}$. 
The $P$--quadratic module $\left(\Der_{K}(P,P),\tilde h_{f}\right)$ depends solely on $f$,
not on the chosen coordinate system $x_{i}$. 
\end{sit}

\begin{sit}
If we take a family of polynomials $\bff= \{f_{b}\}_{b\in Y}$, then the associated (linear) {\em Jacobian}, respectively 
the (quadratic) {\em Hessian map}
are given by
\[
\widetilde{\jac_{\bff}}= \sum_{b\in Y}\widetilde{\jac_{f_{b}}}\partial_{f_{b}}\,,\quad
\tilde h_{\bff}= \sum_{b\in Y}\tilde h_{f_{b}}\partial_{f_{b}}:\Der_{K}(P,P)\to \prod_{b\in Y}P\partial_{f_{b}}\,,
\]
where the symbols $\partial_{f_{b}}$ simply represent the corresponding $\delta$--functions in the product of
the free $P$--modules $P\partial_{f_{b}}$.
\end{sit}

\begin{sit}
\label{sit:2.4}
If $A=P/I$, with $I=(f_{b})_{b\in Y}$, then $A$ is a commutative $K$--algebra and the module of   
K\"ahler differentials of $A$ over $K$ has a free $A$--module presentation
\begin{align*}
\bigoplus_{b\in Y}A[f_{b}]\xto{\ j\ }A\otimes_{P} \Omega^{1}_{P/K}\cong \bigoplus_{x\in X}Adx
\xto{\ \hphantom{\cong}\ }
\Omega^{1}_{A/K}\xto{\ \hphantom{\cong}\ } 0\,,
\end{align*}
where the basis element $[f_{b}]$ is mapped to $j([f_{b}])=1\otimes df_{b}$. The $A$--linear map $j$
factors as the surjection $\bigoplus_{b\in Y}A[f_{b}]\to I/I^{2}$ onto the {\em conormal module\/} 
$I/I^{2}$ that sends $[f_{b}]$ to $f_{b}\bmod I^{2}$, followed by the $A$--linear map 
$d\colon I/I^{2}\to A\otimes_{P} \Omega^{1}_{P/K}$ that is induced by the universal 
$K$--derivation, sending $f\bmod I^{2}$ to $1\otimes df$.

It follows that the Jacobian and Hessian descend,
\begin{align}
\label{jacandh}
\tag{$*$}
\jac_{\bff} &= \sum_{b\in Y}\jac_{f_{b}}\partial_{f_{b}}\,,\quad
h_{\bff}= \sum_{b\in Y} h_{f_{b}}\partial_{f_{b}}:\Der_{K}(P,A)\to \prod_{b\in Y}A\partial_{f_{b}}\,,
\end{align}
the $A$--linear map $\jac_{\bff}$ taking its values in $N_{A/P}=\Hom_{A}(I/I^{2},A)$, the 
{\em normal module\/} of $A$ with respect to $P$, viewed as a submodule of  
$\prod_{b\in Y}A\partial_{f_{b}}$, the $A$--dual of the free $A$--module $\bigoplus_{b\in Y}A[f_{b}]$.
The kernel of the Jacobian satisfies  
\[
\Ker\jac_{\bff}\cong \Der_{K}(A,A)\cong \Hoch^{1}(A/K,A)\cong \Ext^{1}_{A^{\ev}}(A,A)\,.
\]  
Note that the last isomorphism holds for any associative $K$--algebra $A$, when $A^{ev}$ denotes the 
enveloping algebra of $A$ over $K$, while in higher degrees one only has
natural maps $\Hoch^{i}(A/K,A)\to  \Ext^{i}_{A^{\ev}}(A,A)$ that are isomorphisms when $A$ is 
projective as $K$--module; see also \ref{HochExt} below.
\end{sit}

Now we can formulate our first result.
\begin{theorem}
\label{ThmA}
Let $A=P/I$ be a presentation of a commutative $K$--algebra as quotient of a polynomial ring $P=K[x; x\in X]$ over 
some commutative ring $K$ modulo an ideal $I=(f_{b})_{b\in Y}\subseteq P$. 
\begin{enumerate}[\rm(1)]
\item The Hessian map $h_{\bff}$ induces a well defined $A$--quadratic map
\begin{align*}
q:\Der_{K}(A,A)&\ \lto\  \Hom_{A}(I/I^{2},A)
\intertext{from the module of derivations to the normal module. Explicitly,}
q\left(\sum_{x\in X}a_{x}\frac{\partial}{\partial x}\right)(f \bmod I^{2}) &= 
\sum_{1\leqslant i\leqslant j\leqslant n}\frac{\partial^{(2)}f(\bx)}{\partial x_{i}\partial x_{j}}a_{i}a_{j}\bmod I\,,
\end{align*}
where, as above, $x_{1},...,x_{n}\in X$ are the variables that
actually occur in a presentation of $f$, linearly ordered in some way.
\item
There is a canonical $A$--linear map
\begin{align*}
\delta:  \Hom_{A}(I/I^{2},A)\lto \Ext^{2}_{A^{\ev}}(A,A)\,,
\end{align*}
whose image is $T^{1}_{A/K}\subseteq \Ext^{2}_{A^{\ev}}(A,A)$, the first {\em tangent\/} or 
{\em Andr\'e--Quillen cohomology\/} of $A$.
\item
The map from $\Der_{K}(A,A)=\Ext^{1}_{A^{\ev}}(A,A)$ to $\Ext^{2}_{A^{\ev}}(A,A)$ that sends a 
derivation $D$ to the class $[D\circ D]$ of its square under Yoneda product   
equals $\delta\circ q$ and takes its image in the $2$--torsion of $T^{1}_{A/K}$.
\end{enumerate}
\end{theorem}

\begin{remark}
\label{HochExt}\quad
If $A$ as in the preceding result is further {\em projective\/} over $K$, then 
there is a natural bijection $\Hoch^{\bdot}(A/K,A)\to \Ext^{\bdot}_{A^{\ev}}(A,A)$ of graded 
commutative $K$--algebras%
\footnote{Simply using the definition of the algebra structure on
source or target, one obtains an isomorphism from the {\em opposite algebra\/} of 
$\Hoch^{\bdot}(A/K,A)$ to the Yoneda $\Ext$--algebra. However, as already Gerstenhaber 
pointed out in his original paper, sending an element $u$ of degree $m$ to $(-1)^{\binom{m}{2}}u$
provides an algebra isomorphism between a graded commutative algebra and its opposite 
algebra.}. 

Moreover, $T^{1}_{A/K}\subseteq \Hoch^{2}(A/K,A)$ identifies with the $A$--submodule of  classes of 
{\em symmetric\/} Hochschild $2$--cocycles.
\end{remark}

\begin{sit}
In our second result, let the $K$--algebra $A=P/I$ be presented as above, but assume moreover that
\begin{enumerate}[\rm (a)]
\item $A$ is transversal to itself as $K$--module in that 
\[
\Tor^{K}_{+}(A,A) =\oplus_{i>0}\Tor^{K}_{i}(A,A)=0\,.
\]
\item The ideal $I\subseteq P$ is {\em regular\/} in the sense of Quillen \cite[Def.6.10]{Qui}, that is, $I/I^{2}$ 
is a projective $A$--module and the canonical morphism of graded commutative algebras
\begin{align*}
{\bigwedge}_{A}(I/I^{2})\lto \Tor^{P}(A,A)
\end{align*}
is an isomorphism. In ({\em loc.cit.\/}) Quillen points out that  for $P$ noetherian, $I$ is regular in $P$ if, and only if,
it is generated locally by a $P$--regular sequence.
\end{enumerate}
We will say, for short, that $A$ is a {\em homological complete intersection algebra\/} over $K$ if 
conditions (a) and (b) are satisfied. Alternative terminology would be to say that the kernel of the 
multiplication map $\mu:A^{\ev}\to A$ has ``free'' or ``exterior Koszul homology,'' a notion introduced and studied 
by A.~Blanco, J.~Majadas Soto, and A.~Rodicio Garcia in a series of articles following \cite{BMR}, or to 
call the multiplication map a ``quasi-complete intersection homomorphism'' as in \cite{ABS}.
\end{sit}
\begin{sit}
The key properties of such homological complete intersection algebras are that the corresponding 
{\em Tate model\/} is small and the associated {\em cotangent complex\/} $\LL_{A/K}$ is, 
by \cite[Cor.~6.14]{Qui}, concentrated in homological degrees $0$ and $1$. More precisely,
\begin{align*}
\xymatrix{
\LL_{A/K}\ \equiv\ 0\ar[r]& I/I^{2}[1]\ar[r]^-{d}&A\otimes_{P}\Omega^{1}_{P/K}[0]\ar[r]&0
}
\end{align*}
and its shifted $A$--dual, the {\em shifted tangent complex\/} $\mathfrak t_{A}:=\Hom_{A}(\LL_{A/K}[1],A)$ is the 
complex 
\begin{align*}
\xymatrix{
\mathfrak t_{A}\ \equiv\ 0\ar[r]& \Der_{K}(P,A)[-1]\ar[r]^-{\jac_{\bff}}&N_{A/P}[-2]\ar[r]&0\,,
}
\end{align*}
concentrated in cohomological degrees one and two.
\end{sit}

\begin{sit}
\label{Qui}
The link between tangent and Hochschild cohomology was established in Quillen's fundamental 
paper \cite[Sect. 8]{Qui}. In essence, it says that the complex returning Hochschild homology can 
be realized as a DG Hopf algebra whose primitive part is represented by the cotangent complex. 
The qualification here is two-fold: there generally is only a spectral sequence reflecting this 
interpretation that generally only degenerates in characteristic zero, and, on the dual side, 
for Hochschild cohomology vis-\`a-vis the enveloping algebra of the tangent complex, the 
algebra structures in cohomology do not match up. However, for complete intersection algebras, 
things are well controlled because the cotangent complex is so short. In detail, we can 
complement the picture as follows.
\end{sit}

\begin{sit}
The Hessian $A$--quadratic map $q$ from \ref{sit:2.4}(\ref{jacandh}) above can be interpreted
as defining on $\mathfrak t_{A}$ the structure of a graded (super) Lie algebra, where we follow 
the axiomatization of this notion in \cite{Avr1}, and the Jacobian can be viewed as a Lie algebra 
differential. The enveloping algebra of this DG Lie $A$--algebra; see \cite{Sjo}; is readily identified 
with the {\em Clifford algebra\/} $\Cliff(q)$ on that quadratic map and the Jacobian induces a DG 
algebra differential $\partial_{\jac_{\bff}}$ on it. 
In particular, the cohomology of this DG algebra inherits an algebra structure this way.
\end{sit}

\begin{theorem}
\label{thm:ci}
Let $A=P/I$ be a homological complete intersection algebra over $K$, with $P=K[x_{1},..., x_{n}]$
and $I=(f_{1},..., f_{c})$ such that the $f_{j}$ induce an $A$--basis of the $A$--module $I/I^{2}$ that we 
hence further assume to be free.

The cohomology of the DG algebra  $(\Cliff(q), \partial_{\jac_{\bff}})$ 
is then the Yoneda $\Ext$--algebra $\Ext^{\bdot}_{A^{\ev}}(A,A)$ of self extensions of $A$ as
bimodule. In detail, 
\begin{align*}
\Cliff(q) &\cong A\langle t_{1},..., t_{n}; s_{1},..., s_{c}\rangle\,,
\end{align*}
with the $t_{i}$ in degree $1$, the $s_{j}$ in degree $2$, the differential determined by
\begin{align*}
\partial(t_{i})&=\sum_{j=1}^{c}\left(\frac{\partial f_{j}}{\partial x_{i}} \bmod I \right) s_{j}\,,\quad \partial(s_{j})=0\,.
\end{align*}
With respect to the multiplicative structure, the $s_{j}$ are central and
\begin{align*}
t_{i}^{2}&= \sum_{j=1}^{c}\left(\frac{\partial^{(2)}f_{j}}{\partial x_{i}^{2}}\bmod I\right)s_{j}\\
t_{i}t_{i'}+t_{i'}t_{i}&= \sum_{j=1}^{c}\left(\frac{\partial^{2}f_{j}}{\partial x_{i}\partial x_{i'}}\bmod I\right)s_{j}\,.
\end{align*}
If $A$ is projective over $K$, then the natural homomorphism of graded commutative algebras
$\Hoch^{\bdot}(A/K)\to \Ext^{\bdot}_{A^{\ev}}(A,A)$ is an isomorphism, thus, the Hochschild cohomology is given
as cohomology of that DG Clifford algebra too.
\end{theorem}

\begin{remark}
In concrete terms, expanding on \ref{Qui} above,
the Poincar\'e--Birkhoff--Witt theorem for enveloping algebras of (super) Lie algebras
shows that as differential graded {\em coalgebras}, in particular, as complexes of graded $A$--modules,
one has that
$\left(\Cliff(q),\partial_{\jac_{\bff}}\right)$ is isomorphic to the {\em Koszul complex\/} $\KK$ over 
$\SS=\Sym_{A}(N_{A/P})\cong A[s_{1},...,s_{j}]$ on the sequence
\begin{gather*}
\sum_{j=1}^{c}\left(\frac{\partial f_{j}}{\partial x_{i}}\bmod I\right)s_{j}\in \SS\,;\quad  i=1,..., n.
\end{gather*}
Bigrading the Koszul complex by placing the variables $t_{i}=\partial_{x_{i}}\in \Der_{K}(P,A)$ into
bidegree $(1,0)$ and $s_{j}=\partial_{f_{j}}$ into bidegree $(0,1)$, one regains the 
-- well-known \cite{Wol, GG} -- $A$--linear {\em Hodge decomposition\/}  of the Hochschild cohomology
for such a homological complete intersection,
\begin{align*}
\Hoch^{p}(A/K) &\cong \bigoplus_{i+2j =p} H^{i}(\KK)_{j}\\
H^{i}(\KK)_{j} &\cong H^{i+2j}\left(\Hom_{A}(\SS_{i+j}(\LL_{A/K}[1]), A)\right)\,.
\end{align*}
However, the preceding theorem says that, generally, this decomposition is not compatible with the 
multiplicative structure, in that squaring sends $\Hoch^{1}(A/K) = H^{1}(\KK)_{0}$ to
\begin{align*}
H^{0}(\KK)_{1}&\cong T^{1}_{A/K}=H^{2}(\Hom_{A}(\LL_{A/K}[1], A))\subseteq \Hoch^{2}(A/K)\,, 
\intertext{but {\em not\/} to the other summand}
H^{2}(\KK)_{0}& \cong \Hom_{A}(\Omega^{2}_{A/K},A)=H^{2}(\Hom_{A}(\SS_{2}(\LL_{A/K}[1]), A))
\subseteq \Hoch^{2}(A/K)\,.
\end{align*}
In other words, the ``obvious'' multiplication on the Hodge decomposition needs to be twisted 
or ``quantized'', here from the exterior multiplicative structure of the Koszul complex to that of 
the (homogeneous) Clifford algebra. 

This phenomenon is analogous; see \cite{CvdB} and the discussion therein;  to the situation 
for non-affine smooth schemes in characteristic zero. 
Similarly, for arbitrary commutative rings over a field of characteristic zero, the relation between
Hodge decomposition and multiplicative structure on Hochschild cohomology is studied from
a combinatorial point of view in \cite{BW}.

We hasten to point out that for the affine (homological) complete intersections considered here this 
multiplicative twist can only make a difference when $2$ is not a unit in $A$.
\end{remark}

We end the paper by presenting some simple examples that highlight how the treatment here encompasses
in a unified manner various results scattered throughout the literature.

\section{Tate Models for the Multiplication Map of an Algebra}

Let $S\to R$ be a homomorphism between commutative rings.
In his seminal paper \cite{Ta}, John Tate presented in essence the following construction of a free
$S$--resolution of $R$; see also \cite{Avr2} for further details, and, in particular, \cite{SemCartan} for 
details on differential graded (DG) algebras with divided powers. The special case of the multiplication 
map $R\otimes_{K}R\to R$ for a commutative $K$--algebra $R$ of finite type that we will be 
mainly interested in has also been detailed in \cite{AvVi}.

\begin{sit}
There exists a factorization of the given ring homomorphism as $S\to (\calt,\partial)\xto{\pi} R$ such that
\begin{enumerate}[\rm (i)]
\item
$\calt$ is a strictly graded commutative $S$--algebra of the form
\begin{align*}
\calt_{\bdot}=\Sym_{S}(F_{0})\otimes_{S}{\bigwedge}_{S}\left(\bigoplus_{i\ \text{odd}}F_{i}\right)\otimes_{S}
\Gamma_{S}\left(\bigoplus_{i>0\ \text{, even}}F_{i}\right)\,,
\end{align*}
with divided powers (in even degrees), graded so that $S$ is concentrated in degree $0$, each 
$F_{i}, i\geqslant 0,$ a free $S$--module concentrated in degree $i$. 

Here $\Sym$ denotes the symmetric algebra functor, $\bigwedge$ the exterior algebra functor, 
and $\Gamma$ the divided power algebra functor.

\item
The $S$--algebra differential $\partial:\calt_{\bdot}\to \calt_{\bdot -1}$ 
is compatible with divided powers, thus, uniquely determined by its restriction to 
$\bigoplus_{i\geqslant 1}F_{i}\subseteq\calt$, vanishing necessarily on $F_{0}$ for degree reasons.

\item
The morphism $\pi$ is a quasiisomorphism  of DG algebras, where one endows $R$, concentrated 
in degree zero, with the, only possible, zero differential. 
\end{enumerate}
Such factorization of $S\to R$ is then a {\em Tate model\/} for $R$ over $S$, but, with the usual 
slight abuse, we also simply call  $(\calt,\partial)$ a Tate model, the remaining data understood. 

For given $i$, the module $\oplus_{j\leqslant i}F_{j}$ generates a DG subalgebra $\calt^{(i)}$ 
with divided powers, and 
$\calt^{(0)}\subseteq\cdots\subseteq \calt^{(i)}\subseteq \calt^{(i+1)}\subseteq 
\cdots \subseteq \calt = \bigcup_{i}\calt^{(i)}$ constitutes the 
associated {\em Postnikov tower}. It follows from the description given that the inclusion 
$\calt^{(i)}\to \calt$ induces isomorphisms $H_{j}(\calt^{(i)})\xto{\ \cong\ }H_{j}(\calt)$ for $j<i$ 
and that $\partial|F_{i+1}:F_{i+1}\to \calt^{(i)}$ can be interpreted as an $S$--linear 
{\em attaching map\/} that kills $H_{i}(\calt^{(i)})$.
\end{sit}

\begin{remark}
Viewed just as a complex, such a Tate model constitutes in particular a free resolution of $R$ as 
$S$--module, thus, can be used to calculate $\Tor^{S}(R,-)$ and $\Ext_{S}(R,-)$.

If $S$ is noetherian and $S\to R$ is surjective, one may choose $F_{0}=0$ and each $F_{i}$, 
for $i\geqslant 1$, to be free of finite rank over $S$, and then $\calt$ is in turn free of finite rank 
over $S$ in each degree.
If $R$ is only of finite type as an $S$--algebra, then one may choose $F_{0}$ to be 
a finite free $S$--module and each $F_{i}$ a finite free $\calt^{(0)}=\Sym_{S}(F_{0})$--module. 
\end{remark}

\begin{sit}
Now consider the particular case, where $A$ is a commutative $K$--algebra over some 
commutative ring $K$, and $\mu\colon A^{\ev}:=A\otimes_{K}A=S\to A=R$ is the 
multiplication map, a $K$--algebra homomorphism from the enveloping algebra of $A$ over $K$ to $A$.
We denote $\tOmega^{1}_{A}=\tOmega^{1}_{A/K}$ the kernel of the multiplication map and note 
that this ideal in $A^{\ev}$ represents the functor of taking $K$--linear derivations into arbitrary 
$A$--bimodules --- in contrast to the $A$--module of {\em K$\ddot{a}$hler differential forms\/} 
$\Omega^{1}_{A/K}= \tOmega^{1}_{A}/(\tOmega^{1}_{A})^{2}$ 
that represents the  functor of taking $K$--linear derivations in {\em symmetric\/} $A$--bimodules.
We use the convention that the universal derivation $d:A\to \tOmega^{1}_{A}$ sends 
$a\mapsto da = 1\otimes a - a\otimes 1$.
\end{sit}

Assume now that $A$ is presented as $(\Sym_{K}F)/I$, for some free $K$--module $F$ 
and some ideal $I$ in the polynomial ring  $P=\Sym_{K}F$ over this free module. 

The following result holds for an arbitrary presentation $A=P/I$, with neither $P$ nor $A$ necessarily
commutative.
\begin{lemma}[cf. {\cite[Cor.2.11]{CuQui}}]
\label{ZJsequence}
There exists an exact sequence of $A^{\ev}$--modules
\begin{align*}
\label{cotangent}
\tag{$\dagger$}
\xymatrix{
I/I^{2}\ar[r]^-{d}&A\otimes_{P}\tOmega^{1}_{P}\otimes_{P}A\ar[r]^-{j}&A^{\ev}\ar[r]^{\mu}&
A\ar[r]&0
}
\end{align*}
where $d$ is induced by restricting the universal $K$--linear derivation $d:P\to \tOmega^{1}_{P}$
to $I$ and $j$ is obtained from the inclusion $i\colon \tOmega^{1}_{P}\subseteq P\otimes_{K} P$
as $j\cong A\otimes_{P}i\otimes_{P}A$ taking into account $A^{\ev}\cong A\otimes_{P}(P\otimes_{K}P)\otimes_{P}A$.\qed
\end{lemma}

Returning to the commutative situation, choose a $P^{\ev}$--linear surjection 
$P\otimes_{K}G\otimes_{K}P\to I$  for some free $K$--module $G$. The preceding Lemma 
has the following immediate consequence.

\begin{cor}
\label{cor:tate}
One may construct a Tate model for the multiplication map $\mu:A\otimes_{K}A\to A$
with $F_{0}=0, F_{1}= A\otimes_{K}F\otimes_{K}A$, and $F_{2} = A\otimes_{K}G\otimes_{K}A$.
\end{cor}

\begin{proof}
As the multiplication map is surjective, one may choose $F_{0}=0$ in the construction of the Tate model.

The image of the restriction of the universal derivation $d:P\to \tOmega^{1}_{P}$ to $F\subseteq P$ 
generates the ideal $\tOmega^{1}_{P}\subseteq P\otimes_{K}P$ and this shows that
the induced map $A\otimes_{K}F\otimes_{K}A\to A\otimes_{P}\tOmega^{1}_{P}\otimes_{P}A$
is surjective, thus, the choice of $F_{1}$. The resulting DG algebra $\calt^{(1)}=\bigwedge_{A^{\ev}}F_{1}$
is nothing but the Koszul complex over the composition 
$A\otimes_{K}F\otimes_{K}A\to A\otimes_{P}\tOmega^{1}_{P}\otimes_{P}A\xto{j} A^{\ev}$
whose first homology is isomorphic to $d(I/I^{2})$ --- this is exactness of  the sequence (\ref{cotangent}) 
above at the term $A\otimes_{P}\tOmega^{1}_{P}\otimes_{P}A$.

Applying on both sides the tensor product over $P$ with $A$ to the
$P^{\ev}$--linear surjection $P\otimes_{K}G\otimes_{K}P\to I$
yields an $A^{\ev}$--linear surjection first from $F_{2}$ onto $I/I^{2}$ and then via $d$ onto that homology. 
As $F_{2}$ is $A^{\ev}$--projective, one can find a lifting of that surjection into the $1$-cycles of $\calt^{(1)}$ 
and such lifting extends the differential from $\calt^{(1)}$ to 
\[
\calt^{(2)}={\bigwedge}_{A^{\ev}}(F_{1})\otimes_{A^{\ev}}\Gamma_{A^{\ev}}(F_{2})\,,
\]
whence the claim concerning $F_{2}$ follows.
\end{proof}

To be even more specific, Lemma \ref{ZJsequence} and the proof of the corollary indicate how to describe
the possible differentials on $\calt^{(2)}$.

\begin{sit}
\label{lem:taylor}
Let $\bx=\{x_{a}\}_{a}$ be a basis of the free $K$--module $F$. Then the universal derivation sends these
algebra generators to
\[
dx_{a} = x''_{a}-x'_{a}\in \tOmega^{1}_{P}\subseteq P\otimes_{K}P\cong 
\Sym_{K}(F\oplus F)\,,
\] 
where we abbreviate $x'=x\otimes 1, x''=1\otimes x$. Below, we will use the same notation for the respective
residue classes in $A^{\ev}$.

Representing a polynomial $f(\bx)=\sum'_{A=(a_{1},..., a_{k})}\alpha_{A}x_{a_{1}}\cdots x_{a_{k}}\in P$, 
with $\alpha_{A}\in K$, as a finite sum of  (ordered) monomials in (finitely many) basis elements $x_{a}$ of $F$,
the definition of $d$ and then the product rule for this derivation yield
\begin{align}
\label{df}
\tag{$*$}
df &= 1\otimes f -f\otimes 1 = f(\bx'')-f(\bx')\\
\notag
&=\sum_{A=(a_{1},..., a_{k})}\sum_{\kappa=1}^{k}
\alpha_{A}x'_{a_{1}}\cdots x'_{a_{\kappa -1}}dx_{a_{\kappa}} x''_{a_{\kappa+1}}\cdots x''_{a_{k}}\\
\notag
&=\sum_{A=(a_{1},..., a_{k})}\sum_{\kappa=1}^{k}\alpha_{A}
x'_{a_{1}}\cdots x'_{a_{\kappa -1}}(x''_{a_{\kappa}}-x'_{a_{\kappa}})x''_{a_{\kappa+1}}\cdots x''_{a_{k}}
\end{align}
in $\tOmega^{1}_{P}=(x''_{a}-x'_{a})_{a}\subseteq P\otimes_{K}P$.
\end{sit}

\begin{sit}
\label{difT2}
Slightly abusing notation, we may then base the free $A^{\ev}$--module $F_{1}$ in the construction of $\calt$ 
above on symbols $\diff x_{a}$ and define $\partial \diff x_{a} = x''_{a}-x'_{a}\in A^{\ev}$. Similarly, 
if $\{f_{b}\}_{b}$ are polynomials in $I$ whose classes generate $I/I^{2}$, we may base the free 
$A^{\ev}$--module $F_{2}$ on symbols $\diff f_{b}$ and choose
\begin{align}
\label{dfA}
\tag{$**$}
\partial(\diff f_{b}) = \sum_{A=(a_{1},..., a_{k})}\sum_{\kappa=1}^{k}
\alpha_{A}x'_{a_{1}}\cdots x'_{a_{\kappa -1}}\diff x_{a_{\kappa}}
x''_{a_{\kappa+1}}\cdots x''_{a_{k}}\ \in \ F_{1}
\end{align}
for the differential on those basis elements, where the right hand side results from a chosen presentation 
of $f=f_{b}$ as above.

This time, the right hand side will depend on the chosen presentation, but those choices are easily 
controlled, in that one may add
$\partial(\omega_{b})$ to the right hand side for any two-form $\omega_{b}\in \bigwedge^{2}_{A^{\ev}}F_{1}$
 to change the value of $\partial(\diff f_{b})$. These are also the only choices.

If we collect coefficients of the terms $dx_{a}$ in (\ref{df}) for $f=f_{b}$ to write
\begin{align*}
df_{b} &=\sum_{a}\Delta_{b,a}(\bx',\bx'')dx_{a}\,,
\end{align*}
with $\Delta_{b,a}(\bx',\bx'')$ in $P\otimes_{K}P$, then these coefficients satisfy
\begin{align*}
\mu\Delta_{b,a}(\bx',\bx'')&= \Delta_{b,a}(\bx,\bx)=\frac{\partial f_{b}(\bx)}{\partial x_{a}} 
\end{align*}
in $P$ irrespective of the representation we chose.
\end{sit}

\section{Diagonal Approximation and Graded Commutativity}
With $\calt$ a Tate model of the multiplication map $A^{\ev}\to A$ for a commutative $K$--algebra as before,
note that $\calt$ is naturally a complex of  $A$--bimodules so that we can form the tensor 
product $\calt\otimes_{A}\calt$. This tensor product carries again the structure of 
a differential strictly graded commutative algebra with divided powers.

Denote $j_{1,2}\colon\calt\to \calt\otimes_{A}\calt$ the DG algebra homomorphisms $j_{1}(t)=t\otimes_{A}1$,
respectively $j_{2}(t)=1\otimes_{A}t$, for $t\in\calt$.

\begin{sit}
The complex underlying $\calt\otimes_{A}\calt$ is one of free $A$--trimodules, that is, 
free modules over $A^{\otimes 3}$. We make no claim about the homology of this complex, except
to note that the algebra homomorphism $\epsilon\otimes_{A}\epsilon: \calt\otimes_{A}\calt\to A\otimes_{A}A\cong A$
is the natural map onto the zero$^{th}$ homology.

Note, however, that if $A$ is {\em flat\/} over $K$, then $A^{\ev}$ is $A$--flat via both $j_{1}$ or $j_{2}$,
and $\calt$ constitutes a flat resolution of $A$ over $A$ via either $j_{1}$ or $j_{2}$, thus, the 
homology of $\calt\otimes_{A}\calt$ is trivial, equal to $A$ in degree zero.

However, even without flatness we have the following.
\end{sit}

\begin{lemma}
\label{lem:hot}
With notation as introduced above,
\begin{enumerate}[\rm (1)]
\item
\label{epsilon1} The algebra homomorphisms 
\begin{align*}
\epsilon_{1}=\epsilon\otimes_{A}\calt&\colon \calt\otimes_{A}\calt\to A\otimes_{A}\calt\cong \calt\\
\epsilon_{2}=\calt\otimes_{A}\epsilon&\colon \calt\otimes_{A}\calt\to \calt\otimes_{A}A\cong \calt
\end{align*}
are homotopic as morphisms of complexes of $A$--trimodules.
\item 
\label{epsilon2}
If $H_{m}(\calt\otimes_{A}\calt)=0$ for $1\leq m \leq i$ for some $i$, then $\Ker\epsilon_{1}\cap\Ker \epsilon_{2}$
has vanishing homology up to and including degree $i$.

In particular, if $A$ is flat over $K$, then the subcomplex $\Ker\epsilon_{1}\cap\Ker \epsilon_{2}\subseteq \calt\otimes_{A}\calt$ has zero homology.
\end{enumerate}
\end{lemma}

\begin{proof}
As $\epsilon\epsilon_{1}=\epsilon\epsilon_{2}$,
the difference $\epsilon_{2}-\epsilon_{1}$ takes its values in the acyclic complex 
$\Ker(\epsilon\colon\calt\to A)$. As its source is a complex of projective, even free $A$--trimodules, 
the claim (\ref{epsilon1}) follows.

Concerning (\ref{epsilon2}), consider the following commutative diagram of complexes with exact rows and columns
\begin{align*}
\xymatrix{
&0&0&0\\
0\ar[r]&\Ker\epsilon\ar[r]\ar[u]&\calt\ar[r]^-{\epsilon}\ar[u]&A\ar[u]\ar[r]&0\\
0\ar[r]&\Ker\epsilon_{1}\ar[r]\ar[u]&\calt\otimes_{A}\calt\ar[r]^-{\epsilon_{1}}\ar[u]_-{\epsilon_{2}}&
\calt\ar[r]\ar[u]_-{\epsilon}&0\\
0\ar[r]&\Ker\epsilon_{1}\cap\Ker\epsilon_{2}\ar[r]\ar[u]&\Ker\epsilon_{2}\ar[r]\ar[u]&
\Ker\epsilon\ar[r]\ar[u]&0\\
&0\ar[u]&0\ar[u]&0\ar[u]
}
\end{align*}
By construction, $\epsilon$ is a quasi isomorphism, whence $\Ker\epsilon$ is acyclic.
By assumption $\epsilon_{1}, \epsilon_{2}$ are quasiisomorphisms in degrees up to
and including $i$, where we may take $i=\infty$ if $A$ is flat over $K$. The long exact homology sequence
obtained from the short exact sequence, say, at the bottom yields then the claims.
\end{proof}

\begin{defn}[cf.~\cite{San}]
Let $\cala$ be a complex of $A^{\ev}$--modules that comes with an augmentation 
$\cala\xto{\epsilon} A$. A {\em co-unital diagonal approximation\/} on $\cala$ is a morphism 
$\Phi:\cala\to \cala\otimes_{A}\cala$ of $A^{\ev}$--complexes 
so that the following diagram commutes,
\begin{align*}
\xymatrix{
&\cala\ar[d]^-{\Phi}\ar@{=}[ld]\ar@{=}[rd]\\
\cala&\cala\otimes_{A}\cala\ar[l]^-{\epsilon\otimes\cala}\ar[r]_-{\cala\otimes \epsilon}&\cala\,,
}
\end{align*}
where we have identified $\cala\otimes_{A}A\cong\cala\cong  A\otimes_{A}\cala$.
\end{defn}

As is well-known, the existence of a co-unital diagonal approximation on a Tate model 
of the multiplication map implies graded commutativity of the Yoneda product on 
$\Ext_{A^{\ev}}(A,A)$; see, for example, \cite{BFl, San, SA}.

\begin{defn}
Given a co-unital diagonal approximation $\Phi:\calt\to \calt\otimes_{A}\calt$ on the Tate model of
$\mu:A^{\ev}\to A$, define a {\em cup product\/} $\cup = \cup_{\Phi}$ on $\Hom^{\bdot}_{A^{\ev}}(\calt, A)$ through
\begin{align*}
f\cup g = \mu( f\otimes_{A} g)\Phi\colon 
\calt\xto{\ \Phi\ }  \calt\otimes_{A}\calt \xto{\ f\otimes g\ }  A\otimes_{A}A\xto[\cong]{\ \mu\ } A\,,
\end{align*}
for cochains $f,g:\calt\to A$. 
\end{defn}

\begin{proposition}
Let $A$ be a commutative $K$--algebra with a Tate model of $\mu:A^{\ev}\to A$ that 
admits a co-unital diagonal approximation. The corresponding cup product then induces a 
graded commutative product
in $H(\Hom^{\bdot}_{A^{\ev}}(\calt, A)) =\Ext^{\bdot}_{A^{\ev}}(A,A)$. This product coincides with 
the usual Yoneda or composition product on the $\Ext$--algebra.
\end{proposition}

\begin{proof}
Given cocycles $f,g\in \Hom^{\bdot}_{A^{\ev}}(\calt,A)$, lift them to morphisms of complexes
$\tilde f, \tilde g\in \Hom^{\bdot}_{A^{\ev}}(\calt,\calt)$. In the homotopy category of complexes 
$\epsilon_{1},\epsilon_{2}:\calt\otimes_{A}\calt \to \calt$ agree because of Lemma \ref{lem:hot}(\ref{epsilon1}),
and so the following diagram of complexes of $A^{\ev}$--modules commutes in the homotopy category,
\begin{align*}
\xymatrix{
\calt\ar[r]^-{\Phi}\ar@{=}[dr]&\calt\otimes_{A}\calt \ar[r]^{\id\otimes \tilde g}\ar[d]_{\epsilon_{1}}&
\calt\otimes_{A}\calt\ar[d]_{\epsilon_{1}=}^{\epsilon_{2}}\ar[r]^{\tilde f\otimes \id}&
\calt\otimes_{A}\calt\ar[d]^{\epsilon_{2}}
\ar[r]^{\epsilon_{1}\otimes_{A}\epsilon_{2}}&A\otimes_{A}A\ar[dd]_{\cong}^{\mu}\\
&\calt\ar[rrrd]_{f\circ \tilde g}\ar[r]^{\tilde g}&\calt\ar[r]^{\tilde f}\ar[rrd]^{f}&\calt\ar[rd]^-{\epsilon}\\
&&&&A
}
\end{align*}
By definition, $f\circ \tilde g$ is a cocycle representing the Yoneda product of $f$ with $g$, 
while the composition across the top and down yields $f\cup g$. Thus, these two products 
coincide in $\Ext_{A^{\ev}}(A,A)$.

Further, for $\tilde f\in \End^{|f|}_{A^{\ev}}(\calt), \tilde g\in \End^{|g|}_{A^{\ev}}(\calt)$ lifts of homogeneous cocycles 
as above, one has
\[
(\tilde f\otimes \id)(\id\otimes \tilde g)=\tilde f\otimes\tilde g = (-1)^{|f||g|}(\id\otimes\tilde g)(\tilde f\otimes \id)
\]
whence combining the above diagram with that corresponding to $(\id\otimes \tilde g)(\tilde f\otimes \id)$ 
yields graded 
commutativity of the Yoneda or cup product in cohomology, that is, in $\Ext^{\bdot}_{A^{\ev}}(A,A)$.
\end{proof}

\begin{remark}
The preceding proof shows that a partial co-unital diagonal approximation 
$\Phi^{(i)}:\calt^{(i)}\to \calt^{(i)}\otimes_{A}\calt^{(i)}$ is enough to express the Yoneda product of two
cohomology classes whose degrees add up to at most $i$ through the cup product on $\Hom^{\bdot}_{A^{\ev}}
(\calt^{(i)},A)$. Moreover, these classes will commute in the graded sense with respect to the product.
\end{remark}

In the particular case that $A$ is flat over $K$, one can indeed construct a co-unital diagonal 
approximation that is furthermore a homomorphism of DG algebras with divided powers.
\begin{proposition}
\label{prop:diag}
If $A$ is flat over $K$, then there exists a co-unital diagonal approximation $\Phi$ on the Tate model $\calt$ 
of $\mu:A^{\ev}\to A$ that induces such approximation $\Phi^{(i)}$ on each DG algebra $\calt^{(i)}$ in the 
Postnikov tower.
\end{proposition}

\begin{proof}
We construct a desired co-unital diagonal approximation $\Phi$ inductively along the Postnikov tower.
At the basis of the Postnikov tower, we have $\calt^{(0)}= \calt_{0}= A^{\ev}$ and
$\calt^{(0)}\otimes_{A}\calt^{(0)} = (\calt\otimes_{A}\calt)_{0}\cong A^{\otimes 3}\subseteq \calt\otimes_{A}\calt$.
Now
\begin{align*}
\Phi^{(0)}&:\calt^{(0)}\cong A^{\otimes 2}\to \calt^{(0)}\otimes_{A}\calt^{(0)}\cong A^{\otimes 3}\,,\\
\Phi^{(0)}&(a\otimes b)=a\otimes 1\otimes b
\end{align*}
defines a homomorphism of algebras such that 
$\epsilon_{1}\Phi^{(0)}= \id_{A^{\otimes 2}}= \epsilon_{2}\Phi^{(0)}\,.$
Via $\Phi^{(0)}$, we can, and will, view $\calt\otimes_{A}\calt$ as a DG $A^{\ev}$--algebra.
 
Now assume a co-unital diagonal approximation $\Phi^{(i)}:\calt^{(i)}\to \calt^{(i)}\otimes_{A} \calt^{(i)}$
has been defined for some $i\geqslant 0$. With $\calt^{(i+1)} =\calt^{(i)}\langle F_{i+1}\rangle$ for some
free $A^{\ev}$--module $F_{i+1}$, note that 
$(j_{1} + j_{2})|_{F_{i+1}}: F_{i+1}\to \calt^{(i+1)}\otimes_{A}\calt^{(i+1)}$
satisfies 
\[
\epsilon_{1}(j_{1} + j_{2})|_{F_{i+1}}=\id_{F_{i+1}}=\epsilon_{2}(j_{1} + j_{2})|_{F_{i+1}}\,.
\] 
For $m=1,2$, by the  induction hypothesis $\epsilon_{m}\Phi^{(i)}=\id_{\calt^{(i)}}$. As
$\epsilon_{m}$ is a morphism of complexes, it follows that
\[
\epsilon_{m}(\Phi^{(i)}\partial-\partial(j_{1}+j_{2}))=\partial-\partial\epsilon_{m}(j_{1}+j_{2})\,,
\]
vanishes on $F_{i+1}$. Because $\Phi^{(i)}$ is a morphism of complexes,
\begin{align*}
\partial(\Phi^{(i)}\partial-\partial(j_{1}+j_{2})) = 0\,,
\end{align*}
whence it follows that $\Phi^{(i)}\partial-\partial(j_{1}+j_{2})$ maps $F_{i+1}$
into the cycles of $\Ker\epsilon_{1}\cap\Ker\epsilon_{2}$. As the latter complex is acyclic for $A$ flat over 
$K$ by Lemma \ref{lem:hot}(\ref{epsilon2}), and as $F_{i+1}$ is projective, even free, we can find 
$\Phi'_{i+1}:F_{i+1}\to \Ker\epsilon_{1}\cap\Ker\epsilon_{2}$
such that $\Phi^{(i)}\partial = \partial(\Phi'_{i+1}+ j_{1}+j_{2})$ on $F_{i+1}$.

The algebra homomorphism given by $\Phi^{(i)}$ on $\calt^{(i)}$ and extended through
\begin{align*}
\Phi_{i+1} =\Phi'_{i+1} + j_{1}+ j_{2}
\end{align*}
first to $F_{i+1}$ and, then as algebra homomorphism to $\Phi^{(i+1)}$ on $\calt^{(i+1)}$ 
satisfies $\partial\Phi^{(i+1)}=\Phi^{(i+1)}\partial$ and $\epsilon_{m}\Phi^{(i+1)}=\id$.
In other words, we have found an extension of $\Phi^{(i)}$ to a co-unital diagonal approximation on $\calt^{(i+1)}$,
completing the inductive argument.
\end{proof}

\section{Products in Low Degree}
\begin{sit} We can sharpen Proposition \ref{prop:diag} a bit in low degrees in that we do not 
need any flatness there.

Proceding as in the proof of Proposition \ref{prop:diag}, define 
$\Phi_{1}=( j_{1}+j_{2})|_{F_{1}}:F_{1}\to \calt^{(1)}\otimes_{A}\calt^{(1)}$ and 
extend $\Phi^{(0)}, \Phi_{1}$ to a homomorphism  
$\Phi^{(1)}\colon \calt^{(1)}\to \calt^{(1)}\otimes_{A}\calt^{(1)}$ of strictly commutative graded 
algebras with divided powers. It will indeed be a co-unital approximation and homomorphism 
of DG algebras, as $\Phi^{(0)}$ is such a homomorphism and
\begin{align*}
\partial \Phi_{1}(\diff x_{a}) &= \partial( j_{1}+j_{2})(\diff x_{a}) =( j_{1}+j_{2}) \partial(\diff x_{a}) \\
&=  j_{1}(1\otimes x_{a}-x_{a}\otimes 1) + j_{2} (1\otimes x_{a}-x_{a}\otimes 1)\\
&= (1\otimes x_{a}\otimes 1 - x_{a}\otimes 1\otimes 1) +
(1\otimes 1\otimes x_{a}-1\otimes x_{a}\otimes 1)\\
&=1\otimes 1\otimes x_{a}-x_{a}\otimes 1\otimes 1=\Phi^{(0)}(\partial \diff x_{a})
\end{align*}
for any basis element $\diff x_{a}$ in $F_{1}$.

Put differently, a diagonal approximation $\Phi^{(1)}:\calt^{(1)}\to \calt^{(1)}\otimes_{A}\calt^{(1)}$ 
can be chosen in such a way that
\begin{align*}
\Phi^{(1)}(a\otimes b)&= a\otimes 1\otimes b &&\text{for $a\otimes b\in A^{\ev}$,}\\
\Phi^{(1)}(\diff x_{a})&= \diff x'_{a} + \diff x''_{a}&&\text{for $\diff x_{a}\in F_{1}$,}
\end{align*}
where we have shortened notation to $ \diff x'_{a}= \diff x_{a}\otimes 1=j_{1}(\diff x_{a})$ and 
$\diff x''_{a} = 1\otimes \diff x_{a}=j_{2}(\diff x_{a})$. 
We also use the algebra homomorphism $\Phi^{(0)}: A^{\ev}\to A^{\otimes 3}$ to identify
\begin{align*}
x' =\Phi^{(0)}(x') = x\otimes 1\otimes 1\,,\quad x = 1\otimes x\otimes 1\,,\quad x'' =\Phi^{(0)}(x'') = 1\otimes 1\otimes x\,.
\end{align*}
\end{sit}

\begin{sit}
\label{deg1}
It is noteworthy that the induced cup product on $\Hom^{\bdot}_{A^{\ev}}(\calt^{(1)},A)$ is always
{\em strictly\/} graded commutative. Indeed, recall that $\calt^{(1)}$ is nothing but the Koszul complex on 
$(x''_{a}-x'_{a})_{a}$ in $A^{\ev}$.
Dualizing into $A$, the induced differential becomes zero, and there is a quasi isomorphism
\begin{align*}
\Hom^{\bdot}_{A^{\ev}}(\calt^{(1)},A) &
\simeq \bigoplus_{i\geqslant 0}\Hom_{A^{\ev}}(A\otimes_{P}\Omega^{i}_{P/K}\otimes_{P}A, A) \cong 
\bigoplus_{i\geqslant 0}\Hom_{P}(\Omega^{i}_{P/K}, A) 
\end{align*}
to the graded module of alternating $K$--linear polyvector fields on $P$ with values in $A$.
If now $f\in \Hom^{|f|}_{A^{\ev}}(\calt^{(1)},A)$ is a cocycle of {\em odd degree}, then $f\cup f=0$. In fact, if $\omega\in 
\calt^{(1)}_{2|f|}$, then $\Phi^{(1)}(\omega)$, in Sweedler notation, is of the form
\begin{align*}
\Phi^{(1)}(\omega) =\sum (\omega_{(1)}\otimes_{A} \omega_{(2)} +
(-1)^{|\omega_{(1)}||\omega_{(2)}|}\omega_{(2)}\otimes_{A}\omega_{(1)})
\end{align*}
for suitable homogeneous $\omega_{(k)}\in \calt^{(1)}, k=1,2,$ whose degrees add up to $2|f|$, as the 
explicit form of $\Phi^{(1)}$ shows. 
Now $f\otimes f$ is zero on any summand with $|\omega_{(k)}|\neq |f|$, while on summands where  
 $|\omega_{(k)}|= |f|$ is odd, one has 
 \[
 (f\otimes f)(\omega_{(1)}\otimes_{A} \omega_{(2)} - \omega_{(2)}\otimes_{A}\omega_{(1)}) = 
 f(\omega_{(1)})f(\omega_{(2)})- f(\omega_{(2)})f(\omega_{(1)})=0\,,
 \]
 as $A$ is commutative.
\end{sit}

\begin{sit}
As just demonstrated, the construction of $\Phi^{(1)}$ does not require $A$ to be flat over $K$, 
and we now show by explicit construction that one can find also without any 
flatness assumption a co-unital diagonal approximation 
$\Phi^{(2)}$ on $\calt^{(2)}$ that extends the just constructed $\Phi^{(1)}$.

Recall that $\calt^{(2)}=\calt^{(1)}\langle F_{2}\rangle$, with $F_{2}$ a free $A^{\ev}$--module based on 
elements $\diff f_{b}$ that correspond to polynomials $f_{b}\in I\subseteq P$ whose classes in $I/I^{2}$ 
generate that $A^{\ev}$--module. 
An ``attaching map'' $\partial|_{F_{2}}:F_{2}\to \calt^{(1)}$ that killed the second homology of $\calt^{(1)}$
was obtained in \ref{difT2}(\ref{dfA}) from some representation of each $f_{b}$ as $K$--linear 
combination of monomials in the algebra generators of $P$ over $K$.

In those terms, and extending the shorthand notation to 
$\diff f'_{b} = \diff f_{b}\otimes 1 = j_{1}(\diff f_{b})$ and 
$\diff f''_{b} = 1\otimes \diff f_{b} = j_{2}(\diff f_{b})$, we now determine a diagonal approximation in 
degree $2$.
\end{sit}

\begin{proposition}
\label{deg2}
The algebra homomorphism $\Phi^{(2)}:\calt^{(2)}\to \calt^{(2)}\otimes_{A}\calt^{(2)}$ that extends 
$\Phi^{(1)}$ by means of 
\begin{align*}
\Phi_{2}(\diff f_{b}) &= \diff f'_{b}+\diff f''_{b}\\
&-\sum_{A=(a_{1},..., a_{k})}\sum_{1\leqslant \kappa<\lambda\leqslant k}
\alpha_{A}x'_{a_{1}}\cdots x'_{a_{\kappa -1}}\diff x'_{a_{\kappa}}
x_{a_{\kappa+1}}\cdots x_{a_{\lambda-1}}\diff x''_{a_{\lambda}}x''_{a_{\lambda+1}}\cdots x''_{a_{k}}
\end{align*}
for some finite presentation 
$f_{b}(\bx)=\sum'_{A=(a_{1},..., a_{k})}\alpha_{A}x_{a_{1}}\cdots x_{a_{k}}\in P$ of the
given polynomial $f_{b}$ defines a co-unital diagonal approximation. 
\end{proposition}

\begin{proof}
Inspection reveals immediately that $\Phi^{(2)}$ is co-unital, that is,
\[
\epsilon_{1}\Phi^{(2)} = \id_{\calt^{(2)}} = \epsilon_{2}\Phi^{(2)}\,.
\]
It thus remains to verify that $\Phi^{(2)}$ is compatible with the differentials, equivalently, that
$\partial\Phi_{2}(\diff f_{b})=\Phi^{(1)}(\partial\diff f_{b})$ for each $b$.
 
This is a straightforward verification. First, we exhibit $\partial(\diff f'_{b})$ 
and $\partial(\diff f''_{b})$ for use below,
\begin{align*}
\partial(\diff f'_{b})&= \sum_{A=(a_{1},..., a_{k})}\sum_{\kappa=1}^{k}
\alpha_{A}x'_{a_{1}}\cdots x'_{a_{\kappa -1}}\diff x'_{a_{\kappa}} x_{a_{\kappa+1}}\cdots x_{a_{k}}\,,\\
\partial(\diff f''_{b})&=\sum_{A=(a_{1},..., a_{k})}\sum_{\kappa=1}^{k}\alpha_{A}
x_{a_{1}}\cdots x_{a_{\kappa -1}}\diff x''_{a_{\kappa}}
x''_{a_{\kappa+1}}\cdots x''_{a_{k}}\,.
\end{align*}
Next we apply the differential to the double sum in the definition of $\Phi_{2}(\diff f_{b})$ and obtain
\begin{align*}
&\partial\left(\sum_{A=(a_{1},..., a_{k})}\sum_{1\leqslant \kappa<\lambda\leqslant k}
\alpha_{A}x'_{a_{1}}\cdots x'_{a_{\kappa -1}}\diff x'_{a_{\kappa}}
x_{a_{\kappa+1}}\cdots x_{a_{\lambda-1}}\diff x''_{a_{\lambda}}x''_{a_{\lambda+1}}\cdots x''_{a_{k}}\right)\\
&=\sum_{A=(a_{1},..., a_{k})}\sum_{1\leqslant \kappa<\lambda\leqslant k}
\alpha_{A}x'_{a_{1}}\cdots x'_{a_{\kappa -1}}(x_{a_{\kappa}}-x'_{a_{\kappa}})
x_{a_{\kappa+1}}\cdots x_{a_{\lambda-1}}\diff x''_{a_{\lambda}}x''_{a_{\lambda+1}}\cdots x''_{a_{k}}\\
&\ -\sum_{A=(a_{1},..., a_{k})}\sum_{1\leqslant \kappa<\lambda\leqslant k}
\alpha_{A}x'_{a_{1}}\cdots x'_{a_{\kappa -1}}\diff x'_{a_{\kappa}}
x_{a_{\kappa+1}}\cdots x_{a_{\lambda-1}}(x''_{a_{\lambda}}-x_{a_{\lambda}})x''_{a_{\lambda+1}}\cdots x''_{a_{k}}\\
&=\sum_{A=(a_{1},..., a_{k})}\sum_{1\leqslant \lambda\leqslant k}
\alpha_{A}(x_{a_{1}}\cdots x_{a_{\lambda -1}}-x'_{a_{1}}\cdots x'_{a_{\lambda-1}})\diff x''_{a_{\lambda}}
x''_{a_{\lambda+1}}\cdots x''_{a_{k}}\\
&\ -\sum_{A=(a_{1},..., a_{k})}\sum_{1\leqslant \kappa \leqslant k}
\alpha_{A}x'_{a_{1}}\cdots x'_{a_{\kappa -1}}\diff x'_{a_{\kappa}}(x''_{a_{\kappa+1}}\cdots 
x''_{a_{k}}-x_{a_{\kappa+1}}\cdots x_{a_{k}})\\
&=\partial(\diff f''_{b}) + \partial(\diff f'_{b}) -\Phi^{(1)}(\partial \diff f_{b})\,,
\end{align*}
where the first equality uses that $\partial$ is a skew derivation of degree one, and the second equality
applies \ref{lem:taylor}(\ref{df}) to the monomials $x_{a_{1}}\cdots x_{a_{\kappa-1}}$, respectively,
$x_{a_{\kappa+1}}\cdots x_{a_{k}}$, while the final equality uses the form of 
$\partial(\diff f'_{b})$ and $\partial(\diff f''_{b})$ as recalled above.
Reordering the terms, we obtain $\Phi^{(2)}(\partial \diff f_{b}) = \partial\Phi^{(2)}(\diff f_{b})$ as required.
\end{proof}

Note the following special property of the co-unital diagonal approximation just constructed.

\begin{cor}
\label{cocomm}
The co-unital diagonal approximation $\Phi^{(2)}$ is {\em cocommutative\/} in the $\diff f_{b}$, in that
the automorphism of graded algebras 
$\sigma:\calt^{(2)}\otimes_{A}\calt^{(2)}\to \calt^{(2)}\otimes_{A}\calt^{(2)}$,
ignoring differentials, that exchanges $\diff f'_{b}\leftrightarrow \diff f''_{b}$, for each $b$, but leaves the 
remaining variables unchanged, satisfies
$
\sigma\circ\Phi^{(2)} = \Phi^{(2)}
$.\qed
\end{cor}

Having determined one explicit co-unital diagonal approximation on $\calt^{(2)}$, what choices 
were involved? Provided $\calt\otimes_{A}\calt$ is exact in low degrees, 
we can easily describe all co-unital diagonal approximations on $\calt^{(2)}$ that are algebra 
homomorphisms between algebras with divided powers..

\begin{theorem}
\label{diagadapted}
Let $\calt$ be a Tate model of $\mu:A^{\ev}\to A$ as before. The just exhibited co-unital diagonal 
approximation $\Phi^{(2)}$ can be modified to the DG algebra homomorphism 
$\Psi:\calt^{(2)}\to  \calt^{(2)}\otimes_{A}\calt^{(2)}$ that respects divided powers, 
agrees in degree zero with $\Phi^{(0)}$ and is given on the algebra generators of $\calt$ in degrees 
$1$ and $2$ through
\begin{align*}
\Psi(\diff x_{a})&=\diff x'_{a}+ \diff x''_{a} +\partial \omega_{a}\\
\Psi(\diff f_{b}) &= \Phi^{(2)}(\diff f_{b}) + \sum_{a}\Delta_{b,a}(\bx',\bx'')\omega_{a} +\partial \eta_{b}\,,
\end{align*}
where $\Delta_{b,a}\in A^{\ev}$ are as in {\em \ref{lem:taylor}}, the
$\omega_{a}$ are of degree $2$ and the $\eta_{b}$ are of degree $3$ in $\calt^{(2)}\otimes_{A}\calt^{(2)}$, 
such that
$\partial \omega_{a}$ and $\sum_{a}\Delta_{b,a}(\bx',\bx'')\omega_{a} +\partial \eta_{b}$
are in $\Ker\epsilon_{1}\cap\Ker\epsilon_{2}$.

If $\calt^{(2)}\otimes_{A}\calt^{(2)}$ is exact in degrees $1$ and $2$, then these are the only possible 
co-unital diagonal approximations that are homomorphisms of DG algebras with divided powers and 
that agree with $\Phi^{(2)}$ in degree zero. Moreover, in this case $\omega_{a}$ and $\eta_{b}$ 
can already be chosen in $\Ker\epsilon_{1}\cap\Ker\epsilon_{2}$.
\end{theorem}

\begin{proof}
That $\Psi$ is a co-unital diagonal approximation along with $\Phi^{(2)}$ is easily verified directly. 

Conversely, if $\Psi$ is some co-unital diagonal approximation on $\calt^{(2)}$, then
$(\Psi-\Phi^{(2)})(dx_{a})$ is a cycle for each $a$ and, 
if $H_{1}(\calt^{(2)}\otimes_{A}\calt^{(2)})=0$, it is a boundary, thus, of the form $\partial\omega_{a}$ for 
suitable $\omega_{a}$ in degree two. Moreover, $\partial\omega_{a}$ is necessarily a cycle in 
$\Ker\epsilon_{1}\cap\Ker\epsilon_{2}$, as both diagonal approximations are co-unital, and
Lemma \ref{lem:hot}(\ref{epsilon2}) shows that $H_{1}$ 
of that complex vanishes too. Thus, $\omega_{a}$ can be chosen in 
$\Ker\epsilon_{1}\cap\Ker\epsilon_{2}$.

If also $H_{2}(\calt^{(2)}\otimes_{A}\calt^{(2)})=0$, then the $2$--cycles 
$\Psi(df_{b}) - \Phi^{(2)}(df_{b}) - \sum_{a}\Delta_{b,a}\omega_{a}$ are boundaries,
thus, of the form $\partial \eta_{b}$ for some form $\eta_{b}$ of degree $3$. Again,
$\partial \eta_{b}$ belongs necessarily to $\Ker\epsilon_{1}\cap\Ker\epsilon_{2}$ 
and $H_{2}$ of that complex vanishes by the same argument as before, whence we may 
replace $\eta_{b}$,
if necessary, by an element in $\Ker\epsilon_{1}\cap\Ker\epsilon_{2}$ that has the same image 
under the differential.
\end{proof}

To illustrate, we determine the cup product of two $1$--cochains.
\begin{sit}
\label{explicitcup}
Assume $\calt$ is a Tate model of $\mu:A^{\ev}\to A$ for some commutative $K$--algebra $A$.
Let $f,g:\calt\to A$ be $A^{\ev}$--linear cochains of degree $1$. Their cup product is 
then a cochain of degree $2$ with values in $A$, explicitly given by
\begin{align*}
(f\cup g)(\diff x_{a_{1}}\wedge \diff x_{a_{2}})&=\mu(f\otimes_{A}g)\Phi(\diff x_{a_{1}}\wedge \diff x_{a_{2}})\\
&=\mu(f\otimes_{A}g)((\diff x'_{a_{1}} + \diff x''_{a_{1}})\wedge (\diff x'_{a_{2}} + \diff x''_{a_{2}}))\\
&= g(\diff x_{a_{1}})f(\diff x_{a_{2}})-f(\diff x_{a_{1}})g(\diff x_{a_{2}})
\end{align*}
and
\begin{align*}
&(f\cup g)(\diff f_{b})=\mu(f\otimes_{A}g)\Phi(\diff f_{b})\\
&= \sum_{A}\sum_{1\leqslant \kappa<\lambda\leqslant k}
\alpha_{A}x_{a_{1}}\cdots x_{a_{\kappa -1}}f(\diff x_{a_{\kappa}})
x_{a_{\kappa+1}}\cdots x_{a_{\lambda-1}}g(\diff x_{a_{\lambda}})x_{a_{\lambda+1}}\cdots x_{a_{k}}\\
&= \sum_{A}\sum_{1\leqslant \kappa<\lambda\leqslant k}
\alpha_{A}x_{a_{1}}\cdots x_{a_{\kappa -1}}\widehat x_{a_{\kappa}}
x_{a_{\kappa+1}}\cdots x_{a_{\lambda-1}}\widehat x_{a_{\lambda}}x_{a_{\lambda+1}}\cdots x_{a_{k}}f(\diff x_{a_{\kappa}})g(\diff x_{a_{\lambda}})
\end{align*}
\end{sit}

It remains to find a palatable form of the coefficients of $f(\diff x_{a_{\kappa}})g(\diff x_{a_{\lambda}})$ in that expression.

\begin{theorem}
\label{thm:hesse}
Given a commutative $K$--algebra $A$, represented as $A\cong K[x; x\in X]/(f_{b})_{b}$. If for a given $b$, the polynomial
$f_{b}$ depends on the variables $x_{1},..., x_{n}\in X$, one may choose a 
co-unital diagonal approximation $\Phi^{(2)}$ on $\calt^{(2)}$ in such a way that
\begin{align*}
(f\cup g)(\diff f_{b}) &= \sum_{1\leqslant i\leqslant j\leqslant n}
\left(\frac{\partial^{(2)}f_{b}(\bx)}{\partial x_{i}\partial x_{j}}\bmod I\right)f(\diff x_{i})g(\diff x_{j})\ \in\ A\,,
\end{align*}
for any $A^{\ev}$--linear cochains $f,g:\calt\to A$ of degree one.
\end{theorem}

\begin{proof}
Let $f_{b}=\sum_{\be\in \NN^{n}}\alpha_{\be}x_{1}^{e_{1}}\cdots x_{n}^{e_{n}}$ with $a_{\be}\in K$ 
be the presentation of $f_{b}$ as polynomial in
the variables $x_{1},..., x_{n}$. Rewrite the occurring monomials as
\begin{align*}
\alpha_{\be}x_{1}^{e_{1}}\cdots x_{n}^{e_{n}}=\alpha_{\be}x_{a_{1}}\cdots x_{a_{|\be|}}\,,
\end{align*}
where 
\begin{gather*}
x_{a_{1}}=\cdots = x_{a_{e_{1}}} = x_{1}\,,\\
x_{a_{e_{1}+1}}=\cdots = x_{a_{e_{1}+e_{2}}}=x_{2}\,,\\
\quad\vdots\\
x_{a_{e_{1}+\cdots +e_{n-1}+1}}=\cdots =x_{a_{e_{1}+\cdots + e_{n}}} = x_{n}\,.
\end{gather*}
In the corresponding choice of $\Phi_{2}(\diff f_{b})$ in Proposition \ref{deg2}, this term contributes
\begin{align*}
\sum_{1\leqslant \kappa<\lambda\leqslant |\be|}\alpha_{\be}x'_{a_{1}}\cdots x'_{a_{\kappa-1}}\diff x'_{a_{\kappa}} 
x_{a_{\kappa+1}}\cdots x_{a_{\lambda-1}}\diff x''_{a_{\lambda}}x''_{a_{\lambda+1}}\cdots  x''_{a_{|\be|}}\,.
\end{align*}
Specializing to $x_{a}=x'_{a}=x''_{a}$ for each index $a$, and sorting reduces this sum to
\begin{align*}
\alpha_{\be}&
\left(
\sum_{i=1}^{n}\binom{e_{i}}{2} x_{1}^{e_{1}}\cdots x^{e_{i-1}}_{i-1}x^{e_{i}-2}_{i}x^{e_{i+1}}_{i+1}
\cdots x_{n}^{e_{n}}\diff x'_{i}\diff x''_{i}
\right.\\
\ &+\left.\sum_{1\leqslant i<j\leqslant n}e_{i}e_{j} x_{1}^{e_{1}}\cdots x^{e_{i-1}}_{i-1}x^{e_{i}-1}_{i}x^{e_{i+1}}_{i+1}
\cdots x^{e_{j-1}}_{j-1}x^{e_{j}-1}_{j}x^{e_{j+1}}_{j+1}\cdots x_{n}^{e_{n}}\diff x'_{i}\diff x''_{j}\right)
\end{align*}
The coefficient of $\diff x'_{i}\diff x''_{j}$ in this expression is indeed the corresponding divided second derivative of the monomial, 
whence, putting the terms together again, it follows that for this choice of presentation,
\begin{align*}
\Phi_{2}(\diff f_{b})(\bx,\bx,\bx)= \diff f'_{b}+\diff f''_{b} -
\sum_{1\leqslant i\leqslant j\leqslant n}\frac{\partial^{(2)}f_{b}(\bx)}{\partial x_{i}\partial x_{j}}\diff x'_{i}\diff x''_{j}
\end{align*}
and then the cup product takes the form claimed.
\end{proof}

\section{Proof of Theorem \ref{ThmA}}
Let $A=P/I$ be as before a presentation of a commutative $K$--algebra as quotient of a polynomial ring 
$P=K[x; x\in X]$ over some commutative ring $K$ modulo an ideal $I=(f_{b})_{b\in Y}\subseteq P$.
We need to show that the Hessian quadratic map
\begin{align*}
h_{\bff}= \sum_{b\in Y} h_{f_{b}}\partial_{f_{b}}:\Der_{K}(P,A)\to \prod_{b\in Y}A\partial_{f_{b}}
\end{align*}
as defined in \ref{sit:2.4} sends $\Der_{K}(A,A)\subseteq \Der_{K}(P,A)$ to 
$N_{A/P}=\Hom_{A}(I/I^{2},A)\subseteq  \prod_{b\in Y}A\partial_{f_{b}}$.

Now an element $(u_{b}\partial_{f_{b}})_{b}\in\prod_{b\in Y}A\partial_{f_{b}}$, represents an element in
the normal module $N_{A/P}$ if, and only if, for every relation $\sum_{b}z_{b}f_{b}=0$, with $z_{b}\in P$ 
and only finitely many of these elements nonzero, the sum
$\sum_{b}z_{b}u_{b}$ is zero in $A$. 

Let $X'=\{x_{1},..., x_{N}\}\subset X$ be the finite subset of variables that are involved in the finitely
many polynomials $f_{b}$ that occur with nonzero coefficient in the given relation.

If $D=\sum_{x\in X}a_{x}\frac{\partial}{\partial x}\in \Der_{K}(A,A)$ is given, set 
$D'= \sum_{x\in X'}\tilde a_{x}\frac{\partial}{\partial x}$ where $\tilde a_{x}\in P$ lifts $a_{x}\in A$.
One then has $D'(f_{b})\equiv D(f_{b})=0 \bmod I$, and 
\begin{align*}
q(D)_{b} &= \sum_{1\leqslant i\leqslant j\leqslant N}
\frac{\partial^{(2)}f_{b}(\bx)}{\partial x_{i}\partial x_{j}}a_{i}a_{j}\bmod I 
\equiv \tilde h_{\bff}(D')(f_{b})\bmod I
\end{align*}
for those $f_{b}$ that are involved in the relation. Next note that the second divided power $D'^{(2)}=
\sum_{1\leqslant i\leqslant j\leqslant N}\frac{\partial^{(2)}}{\partial x_{i}\partial x_{j}}\tilde a_{i}\tilde a_{j}$ of $D'$
satisfies
\[
D'^{(2)}(fg) = D'^{(2)}(f)g + D'(f)D'(g)+ fD'^{(2)}(g)
\] 
for any polynomials $f,g\in P$.

Apply $D'^{(2)}$ to the relation $0=\sum_{b}z_{b}f_{b}$ to obtain
\begin{align*}
0=D'^{(2)}\left(\sum_{b} z_{b}f_{b}\right)&= \sum_{b}D'^{(2)}(z_{b}f_{b}) \\
&=\sum_{b}(D'^{(2)}(z_{b})f_{b} + D'(z_{b})D'(f_{b}) + z_{b} D'^{(2)}(f_{b}))\\
&\equiv \sum_{b}z_{b} q(D)_{b}\bmod I
\end{align*}
as in the middle equation, each term $D'^{(2)}(z_{b})f_{b}$ is in $I$ and 
$D'(f_{b})\equiv D(f_{b})=0\bmod I$ as $D$ is a derivation on $A$. 
Therefore, $q(D)=(q(D)_{b})_{b}$ is indeed an element of the normal module. 
This proves part (1) of Theorem \ref{ThmA}.

Part (2) of Theorem \ref{ThmA} just restates a well-known fact, but we give another 
deduction using the Tate model $\calt$ of the multiplication map $\mu:A^{\ev}\to A$ and 
establish along the way as well Part (3) of Theorem \ref{ThmA}. 

To this end, recall, from \ref{deg1} above, that $\calt^{(1)}$ is just the Koszul complex on
$(x''_{a}-x_{a})_{a}$ in $A^{\ev}$ and that its dual $\Hom_{A^{\ev}}(\calt^{(1)},A)$ carries the cup 
product induced by $\Phi^{(1)}$ with respect to which its cohomology 
$H^{\bdot}(\Hom_{A^{\ev}}(\calt^{(1)},A)) \cong \Hom_{A}(\Omega^{\bdot}_{P/K},A)$ is a strictly 
graded commutative algebra.

As the diagonal approximation $\Phi$ on $\calt$ extends the diagonal approximation $\Phi^{(1)}$ on $\calt^{(1)}$,
the $A$--linear surjective restriction map 
$p:\calt^{*}=\Hom^{\bdot}_{A^{\ev}}(\calt, A)\to \calt^{(1)*}=\Hom^{\bdot}_{A^{\ev}}(\calt^{(1)}, A)$ 
resulting from the inclusion $\calt^{(1)}\subset \calt$ is compatible with the cup products on source 
and target, thus, induces a $K$--algebra homomorphism in cohomology. 

If $\calk$ is the kernel of the restriction map on the complexes, then the long exact cohomology sequence contains
\begin{align*}
\xymatrix{
H^{1}(\calt^{(1)*})\ar[r]\ar@{=}[d]&H^{2}(\calk)\ar[r]\ar@{=}[d]&
H^{2}(\calt^{*})\ar[r]^-{H^{2}(p)}\ar@{=}[d]&H^{2}(\calt^{(1)*})\ar@{=}[d]\\
\Der_{K}(P,A)\ar[r]^-{\jac}&N_{A/K}\ar[r]^-{\delta}&\Ext^{2}_{A^{\ev}}(A,A)\ar[r]&\Hom_{A}(\Omega^{2}_{P/A},A)\,.
}
\end{align*}
If now $D\in \Der_{K}(A,A)=H^{1}(\Hom^{\bdot}_{A^{\ev}}(\calt,A))$ then $H^{2}(p)(D\cup D) = 0$, as  $H^{\bdot}(p)$
is an algebra homomorphism into a strictly graded commutative algebra. Thus, $D\cup D$ is in the image of the
map $\delta:N_{A/P}\to \Ext^{2}_{A^{\ev}}(A,A)$, and this image is known to be $T^{1}_{A/K}$, 
the first Andr\'e--Quillen cohomology of $A$ over $K$. Moreover, the explicit description of the cup product on
$\Hom_{A^{ev}}(\calt^{(2)},A)$ shows that $q(D)\in N_{A/P}$ is mapped via $\delta$ to $D\cup D$.\qed
 
\section{Hochschild Cohomology of Homological Complete Intersections}
Now we consider the case that $A=P/I$ is a homological complete intersection over $K$ with
$I/I^{2}$ free.

\begin{proposition}{\em (see \cite{BMR} and \cite{ABS, Iy})}
\label{t2}
Assume the commutative $K$--algebra $A$ admits a presentation $A=P/I$, with $P=\Sym_{K}F$ a
symmetric $K$--algebra on a free $K$--module $F$, and $I\subseteq P$ an ideal.   

If $\Tor^{K}_{+}(A,A) =\oplus_{i>0}\Tor^{K}_{i}(A,A)=0$, then $\calt^{(2)}$ as constructed in Corollary
{\em \ref{cor:tate}\/} is already a Tate model of $\mu:A^{\ev}\to A$ if, and only if, 
$I/I^{2}$ is a free $A$--module, the natural homomorphism of strictly graded commutative algebras 
$\bigwedge_{A}I/I^{2}\to \Tor^{P}(A,A)$ is an isomorphism, and the surjection $P\otimes_{K}G\otimes_{K}P
\to I$ was chosen to induce a bijection from a $K$--basis of $G$ to an $A$--basis of $I/I^{2}$.

\end{proposition}

\begin{proof}
In Corollary \ref{cor:tate} we formed $\calt^{(1)}$ as the Koszul complex on $(x''_{a}-x'_{a})_{a}$ in $A^{\ev}$, 
where $x_{a}$ runs through a $K$--basis of $F$. Its homology can be identified as 
\[
 H(\calt^{(1)})\cong \Tor^{P\otimes_{K}P}(P,A\otimes_{K}A)\,.
 \]
By \cite[Chap.~XVI\S 5(5a)]{CE}, flatness of $P$ over $K$ together with $\Tor^{K}_{+}(A,A) = 0$
yields an isomorphism $\Tor^{P\otimes_{K}P}(P,A\otimes_{K}A)\cong \Tor^{P}(A,A)$.

If $I/I^{2}$ is free and the canonical map $\bigwedge_{A}I/I^{2}\to \Tor^{P}(A,A)\cong H(\calt^{(1)})$ is 
an isomorphism of algebras, then already Tate \cite{Ta} showed that $\calt^{(2)}=\calt$, provided the
map $G\to I/I^{2}$ sends a $K$--basis to an $A$--basis.
 The main result of \cite{BMR} yields the converse.
\end{proof}

\begin{remark}
This proposition applies in particular when $I=(f_{1},..., f_{c})\subseteq P=K[x_{1},..., x_{n}]$ defines
a {\em complete intersection}, in the sense that the $f_{j}$ form a Koszul--regular sequence in $P$.
In this case, the condition $\Tor^{K}_{+}(A,A) =0$ is equivalent to exactness of the Koszul complex
on $f_{1}(\bx'),..., f_{c}(\bx'), f_{1}(\bx''_{1}),..., f_{c}(\bx'')$ in $P\otimes_{K}P\cong K[\bx',\bx'']$, in that this 
complex is simply the tensor product of the Koszul complex on $f_{1},..., f_{c}$ in $P$, 
a free resolution of $A$ over $K$, with itself over $K$.

The resolution of a complete intersection ring over its enveloping algebra through the Tate model
has a long history.

Wolffhardt \cite{Wol} was the first to exhibit it, but he apparently was unaware of Tate's result 
and wrote down the resolution in case $K$ is a field of characteristic zero --- in which case divided 
powers can be replaced by symmetric ones at the expense of introducing denominators. 

In explicit form, this resolution was also described in \cite{BKu, BACH, GG}.
As the reviewers of \cite{MRo} in {\sc MathSciNet} note: 
``...the authors recover the calculation of $\Hoch_{*}(A/K,A)$ due originally to K. Wolffhardt 
[Trans. Amer. Math. Soc. 171 (1972), 51Ð66; MR0306192 (46 \#5319)], and more recently 
(since 1988) to a host of other authors, including the reviewers.''

To see that the condition $\Tor^{K}_{+}(A,A) = 0$ is not void even in this special situation, 
contrary to what is implied in \cite{BACH} or \cite{GG}, one may just take $c=1$ with, say, 
$K=\ZZ$, and  $f_{1}=17x\in \ZZ[x]$. See also the next section for a detailed discussion of this point.
\end{remark}

\begin{remark}
Wolffhardt \cite{Wol} treats the analytic case, when $A$ is quotient of a noetherian smooth analytic 
algebra in the broad sense of \cite[Supplement]{And}. We leave it to the readers to retrace the foregoing 
results and to convince themselves that everything works as well in that situation, provided 
one replaces the occurring tensor products by their analytic counterparts.
\end{remark}

We now turn to the

\subsection*{Proof of Theorem  \ref{thm:ci}} In light of \ref{explicitcup} and Theorem \ref{thm:hesse}, 
we know explicitly the form of the cup product on $\Hom_{A^{\ev}}(\calt^{(2)},A)$, thus, on 
$\Hom_{A^{\ev}}(\calt,A)$ in the case at hand.

Set $F^{*}=\Hom_{A^{\ev}}(F,A)$ for any $A^{\ev}$--module $F$.
As a complex, $\Hom^{\bdot}_{A^{\ev}}(\calt^{(2)},A)$ is the Koszul complex on
$\left(\sum_{j=1}^{c}\frac{\partial f_{j}}{\partial x_{i}}{\boldsymbol\partial}_{f_{j}}\right)_{i=1}^{n}$ 
in the polynomial ring $A[{\boldsymbol\partial}_{f_{1}},..., {\boldsymbol\partial}_{f_{c}}]$. 
Here $\{{\boldsymbol\partial}_{f_{j}}\}_{j=1,..,c}\in F_{2}^{*}$ form the dual $A$--basis to 
$\{\diff f_{j}\}_{j=1,.., c}
\subseteq F_{2}$ and we employ the $A$--basis 
${\boldsymbol\partial}_{x_{1}},..., {\boldsymbol\partial}_{x_{n}}
\in F^{*}_{1}$ dual to $\{\diff x_{i}\}_{i=1,...,n}\subseteq  F_{1}$.  
That Koszul complex has then the compact description
\begin{align*}
\Hom^{\bdot}_{A^{\ev}}(\calt^{(2)},A) = 
{\bigwedge}_{A}\left(F^{*}_{1}[-1]\right)\otimes_{A}\Sym_{A}\left(F^{*}_{2}[-2]\right)
\end{align*}
with differential 
$\partial({\boldsymbol\partial}_{x_{i}})=
\sum_{j=1}^{c}\frac{\partial f_{j}}{\partial x_{i}}{\boldsymbol\partial}_{f_{j}}$
and $\partial{\boldsymbol\partial}_{f_{j}}= 0$. The cohomological grading on 
$\Hom^{\bdot}_{A^{\ev}}(\calt^{(2)},A)$ 
is recovered from putting ${\boldsymbol\partial}_{x_{i}}$ into degree $1$ and 
${\boldsymbol\partial}_{f_{j}}$ into degree $2$ as indicated.

It remains to determine the multiplicative structure. 

Identify
$\{{\boldsymbol\partial}_{x_{i_{1}}}\wedge {\boldsymbol\partial}_{x_{i_{2}}}\}_{i_{2}<i_{1}}\subseteq 
\bigwedge^{2}_{A}F^{*}_{1}$ with the $A$--basis dual to 
$\{\diff x_{i_{2}}\wedge \diff x_{i_{1}}\}_{i_{2}<i_{1}}$ in $\bigwedge^{2}_{A}F_{1}$. The 
cup product resulting from the diagonal approximation in Theorem \ref{thm:hesse} is then explicitly given by
\begin{align*}
{\boldsymbol\partial}_{x_{i_{1}}}\cup {\boldsymbol\partial}_{x_{i_{2}}} &=
\begin{cases}
{\boldsymbol\partial}_{x_{i_{1}}}\wedge {\boldsymbol\partial}_{x_{i_{2}}} + 
\sum_{j=1}^{c}\frac{\partial^{(2)}f_{j}}{\partial x_{i_{1}}\partial x_{i_{2}}} {\boldsymbol\partial} f_{b}&
\text{if $i_{1}<i_{2}$,}\\
\sum_{j=1}^{c}\frac{\partial^{(2)}f_{j}}{\partial^{2} x_{i}} {\boldsymbol\partial} f_{j}&\text{if $i=i_{1}=i_{2}$,}\\
{\boldsymbol\partial}_{x_{i_{1}}}\wedge {\boldsymbol\partial}_{x_{i_{2}}} &\text{if $i_{1}>i_{2}$.}
\end{cases}
\end{align*}
Moreover, the elements ${\boldsymbol\partial} f_{j}$ are central, the co-unital diagonal approximation $\Phi$ being 
cocommutative in the $\diff f_{j}$ as pointed out in Corollary \ref{cocomm}. These explicit calculations show that 
$\Hom^{\bdot}_{A^{\ev}}(\calt,A)$ with the cup product $\cup=\cup_{\Phi}$ from Theorem \ref{thm:hesse} 
satisfies the relations
\begin{align*}
{\boldsymbol\partial}_{x_{i}}\cup {\boldsymbol\partial}_{x_{i}} &= 
\sum_{j=1}^{c}\frac{\partial^{(2)}f_{j}}{\partial^{2} x_{i}} {\boldsymbol\partial} f_{j}\,,\\
{\boldsymbol\partial}_{x_{i_{1}}}\cup {\boldsymbol\partial}_{x_{i_{2}}} + 
{\boldsymbol\partial}_{x_{i_{2}}}\cup {\boldsymbol\partial}_{x_{i_{1}}} &= 
\sum_{j=1}^{c}\frac{\partial^{2}f_{j}}{\partial x_{i_{1}}\partial x_{i_{2}}} {\boldsymbol\partial} f_{j}\,,\\
{\boldsymbol\partial}_{x_{i}}\cup {\boldsymbol\partial}_{f_{j}} &= 
{\boldsymbol\partial}_{f_{j}}\cup {\boldsymbol\partial}_{x_{i}}\,,\\
{\boldsymbol\partial}_{f_{j_{1}}}\cup {\boldsymbol\partial}_{f_{j_{2}}} &= 
{\boldsymbol\partial}_{f_{j_{2}}}\cup {\boldsymbol\partial}_{f_{j_{1}}}\,.
\end{align*}
for $i,i_{1},i_{2}=1,..., n$ and $j, j_{1}, j_{2}=1,.., c$.

Putting it together, this shows $(\Hom^{\bdot}_{A^{\ev}}(\calt,A), \cup) \cong \Cliff(q)$ as graded algebras
and the differential on $\Hom^{\bdot}_{A^{\ev}}(\calt,A)$ readily identifies with $\partial_{\jac_{\bff}}$ on 
$\Cliff(q)$.

This completes the proof of Theorem  \ref{thm:ci}.\qed

\section{The Case of One Variable}
\begin{sit}
Here we consider {\em cyclic extensions\/} $A=K[x]/(f)$, with $K$ some commutative ring, 
$f=a_{0}x^{d}+\cdots +a_{d}\in P=K[x]$ a polynomial. We set $I=(f)$ and abbreviate 
$\partial_{x}f=\frac{\partial f}{\partial x}$. Let $\fc=\fc_{f}=(a_{0},..., a_{d})\subseteq K$ 
be the {\em content ideal\/} of $f$. 
By \cite[5.3 Thm.~7]{Nor}, the polynomial $f$ is a {\em non-zero-divisor\/}
in $K[x]$ if, and only if, the annihilator of the content ideal in $K$ is zero, equivalently, 
$\Hom_{K}(K/\fc,K)=0$. 
\end{sit}

To provide some context, we quote the following results from the literature and refer to the references cited below
for unexplained terminology. 

\begin{proposition}
\label{prop:one}
Let $f=a_{0}x^{d}+\cdots +a_{d}\in K[x]$ be a polynomial as above and set $A=K[x]/(f)$. 
\begin{enumerate}[\rm(1)]
\item
\label{part0}
$A$ is {\em free\/} as $K$--module if, and only if, the ideal $I$ is generated by a {\em monic\/} polynomial.
$A$ is then necessarily free {\em of finite rank}.
\item
\label{part1'}
$A$ is {\em finitely generated projective\/} as $K$--module if, and only if,  the ideal $I$ is generated 
by a {\em quasi-monic} polynomial.
\item
\label{part1} 
$A$ is {\em projective\/} as $K$--module if, and only if, the ideal $I$ is generated by an
{\em almost quasi-monic} polynomial. 
\item
\label{part2}
$A$ is {\em flat\/} as $K$--module if, and only if, the content ideal $\fc$ of $f$ is generated in $K$ by an {\em 
idempotent}.
\end{enumerate}
In cases {\em (\ref{part0})\/} and {\em (\ref{part1'})}, $f$ is a non-zero-divisor in $K[x]$, while in cases
{\em (\ref{part1})\/} and {\em (\ref{part2})}, the content ideal is generated by an idempotent $e\in K$,
 $\fc_{f}=(e)$. With $K_{1} = K/(1-e), K_{2} = K/(e)$, so that $K=K_{1}\times K_{2}$,
the polynomial $ef$ is then a non-zero-divisor in $K_{1}[x]$, while $(1-e)f=0$, and 
$A\cong K_{1}[x]/(ef)\times K_{2}[x]$.
\end{proposition}

\begin{proof}
Parts (\ref{part0}), (\ref{part1'}) and (\ref{part1}) can be found in \cite{BM}, where also the terminology 
``(almost) quasi--monic'' is introduced, while part (\ref{part2}) is contained in \cite[Thm.~4.3]{OR}.
\end{proof}

\begin{examples}
The polynomial $f= 15 x^{d} +10x + 6$, for $d\geqslant 2$, is quasi-monic over $K=\ZZ/30\ZZ$.
The polynomial $f= 15 x^{d} + 6$, for $d\geqslant 1$, is almost quasi-monic, but not quasi-monic
over $K=\ZZ/30\ZZ$. The ring $A=\ZZ[x]/(2x-1)\cong \ZZ[\tfrac{1}{2}]$ is flat as module over $K=\ZZ$, 
but not projective.
\end{examples}

\begin{sit}
In each of the cases considered in Proposition \ref{prop:one}, we have a decomposition 
$K=K_{1}\times K_{2}$ and $A\cong A_{1}\times A_{2}$, where $A_{1}=K_{1}[x]/(\overline f)$ with $\overline f$ a 
non-zero-divisor in $K_{1}[x]$ and $A_{2}=K_{2}[x]$ a polynomial ring.
Moreover, as algebras,
\[
\Ext_{A^{\ev}}(A,A)\cong \Ext_{A^{\op}_{1}\otimes_{K_{1}}A_{1}}(A_{1},A_{1})\times 
\Ext_{A_{2}^{\op}\otimes_{K_{2}}A_{2}}(A_{2},A_{2})\,,
\] 
with the second factor satisfying
\[
\Ext_{A_{2}^{\op}\otimes_{K_{2}}A_{2}}(A_{2},A_{2}) \cong K_{2}[x][0]\oplus K_{2}[x]\tfrac{\partial}{\partial x}[-1]
\]
with obvious multiplicative structure. 
\end{sit}
We will therefore now concentrate on the case that the given polynomial
$f\in K[x]$ is a non-zero-divisor.

\begin{sit}
\label{depth2}
If $f$ is a non-zero-divisor, the natural map $A\to I/I^{2}, 1\mapsto f\bmod I,$ is an isomorphism, and
$0\to P\xto{f} P\to A\to 0$ is a free resolution of $A$ both over $P$ and $K$, so that
$\Tor^{K}_{i}(A,A)=0$ for $i\geqslant 2$ and $\bigwedge_{A}I/I^{2}\to \Tor^{P}(A,A)$ is 
trivially an isomorphism.

In other words, $I=(f)\subseteq P$ is a regular ideal. The only missing piece for $A$ to be
a homological complete intersection over $K$ is thus the vanishing of $\Tor^{K}_{1}(A,A)$.
That vanishing is equivalent to $f\otimes 1\in P\otimes_{K} A$ being a non-zero-divisor,
in turn equivalent to $1\otimes f, f\otimes 1$ forming a regular sequence in $P\otimes_{K} P$.

In light of \cite[5.~Exercise 9]{Nor}, the sequence $1\otimes f, f\otimes 1$ is regular in 
$P\otimes_{K} P$ if, and only if, $\Ext^{i}_{K}(K/\fc,K)=0$ for $i=0,1$, that is, the {\em depth\/} of the
content ideal $\fc$ on $K$ {\em is at least\/} $2$.
\end{sit}

\begin{example}
Let $R[s,t,u,v]$ be polynomial ring over a commutative ring $R$,
and set $K=R[s,t,u,v]/(sv-tu)$. For $d\geqslant 1$, the polynomial $f(x) = sx^{d}+v\in K[x]$ has content ideal
$\fc_{f} = (s,v)$ of depth $2$, while $g(x)= sx^{d} +u\in K[x]$ has content ideal $\fc_{g}=(s,u)$ of depth $1$.

Both $f,g$ are thus non-zero-divisors in $K[x]$, but only $f$ defines a homological complete intersection 
$A=K[x]/(f)$ over $K$.
\end{example}

\begin{sit}
\label{cyclictate}
Constructing the beginning of the Tate resolution of the multiplication map 
$\mu:A^{\ev}\cong K[x',x'']/(f(x'),f(x''))\to A, \mu(x')=\mu(x'')=x$, as in Corollary \ref{cor:tate}, one obtains
$\calt^{(2)}$ as a $2$--periodic complex, augmented versus $A$ in degree $0$, 
\begin{align*}
\xymatrix{
0&\ar[l]A&\ar[l]0&\ar[l]0&\ar[l]0&\ar[l]\cdots&\equiv& A\\
0\ar[u]&\ar[u]^{\mu}\ar[l]A^{\ev}&\ar[u]\ar[l]_-{x''-x'}A^{\ev}&\ar[u]\ar[l]_-{\Delta}A^{\ev}&
\ar[u]\ar[l]_-{x''-x'}A^{\ev}&\ar[l]\cdots
&\equiv&\calt^{(2)}\ar[u]_{\epsilon}
}
\end{align*}
where $\Delta\in A^{\ev}$ is the residue class of the  {\em unique\/} polynomial 
$\Delta(x', x'')\in K[x',x'']$ that satisfies $f(x'')-f(x')=(x''-x')\Delta(x',x'')$.
\end{sit}

Summarizing the discussion so far, we have the following characterizations, where we refer to \cite{ABS, BS}
for the terminology and properties of {\em exact zero divisors}. 
\begin{theorem}
\label{hci1}
If $f=a_{0}x^{d}+\cdots +a_{d}$ is a non-zero-divisor in $K[x]$, then the following are equivalent.
\begin{enumerate}[\rm(1)]
\item 
\label{parta}
$A=K[x]/(f)$ is a {\em homological complete intersection\/} over $K$.
\item
\label{partb}
$\calt^{(2)}$ resolves $A$ as $A^{\ev}$--module.
\item
\label{partc}
$(x''-x', \Delta)$ is  a pair of {\em exact zero-divisors\/} in $A^{\ev}$.
\item
\label{partd} 
The content ideal $\fc$ of $f$ has {\em grade at least $2$}, in that $\Hom_{K}(K/\fc, K)=0=\Ext^{1}_{K}(K/\fc,K)$.
\end{enumerate}
Moreover, $A$ will be flat over $K$ if, and only if the content ideal of $f$ is the {\em unit ideal\/} in $K$. 
\end{theorem}

\begin{proof}
The equivalence (\ref{parta})$\Longleftrightarrow$ (\ref{partb}) is a special case of Proposition \ref{t2}, while
(\ref{partb})$\Longleftrightarrow$ (\ref{partc}) is a tautology, in that $(x''-x', \Delta)$ is, by definition, a 
pair of exact zero-divisors in $A^{\ev}$ if, and only if, the ideals $(x''-x')$ and $(\Delta)$ are each 
other's annihilators in $A^{\ev}$, if, and only if, the $2$--periodic complex $\calt^{(2)}$ displayed above is exact.
The equivalence (\ref{partc})$\Longleftrightarrow$ (\ref{partd}) has been discussed in \ref{depth2}
above.

The final claim follows from Proposition \ref{prop:one} in that flatness means that $\fc_{f}$ is generated by an
idempotent. If that idempotent is not $1$, then $f$ is not a non-zero-divisor in $K[x]$.
\end{proof}

Assume for the remainder of this section that $f$ and $A$ satisfy the equivalent conditions in Theorem
\ref{hci1}. We determine the algebra $\Ext^{\bdot}_{A^{\ev}}(A,A)$.

\begin{sit}
Call $A'=K[x]/\left(f(x),\partial_{x} f\right)$, the {\em ramification algebra\/} of $f$ and 
introduce as well the annihilator of $\partial_{x}f$ in $A$, given by 
\[
\theta=\left(f:_{P}\partial_{x}f\right)/(f)\subseteq A\,.
\] 
Note that $\theta(\partial_{x}f)=0$ in $A$, whence $\theta$ is naturally an $A'$--module. 

In fact, $\theta \cong \Hom_{A}(A', A)$, as dualizing the exact sequence
\[
\xymatrix{
0\ar[r]&\theta\ar[r]&A\ar[r]^-{\partial_{x}f}&A\ar[r]&A'\ar[r]&0
}
\]
of $A$--modules shows. The same sequence can be read to identify $\theta\cong\Der_{K}(A,A)$ and 
$A'\cong T^{1}_{A/K}$, the first Andr\'e--Quillen (tangent) cohomology of $A$ over $K$.
\end{sit}

\begin{theorem}
\label{mult}
Assume $f\in K[x]$ is a polynomial whose content ideal $\fc$ is of depth at least two.

The Yoneda $\Ext$--algebra of $A=K[x]/(f)$ as $A^{\ev}$--module is then graded commutative
and satisfies
\begin{align*}
\Ext^{\bdot}_{A^{\ev}}(A,A)\cong A\times_{A'}\frac{(A'\oplus\theta t)[s]}{\left((at)(bt)-abf^{(2)}s;
a,b\in\theta\right)}\,,
\end{align*}
where $A, A'$ and $\theta$ are in degree zero, $t$ is of degree $1$, while $s$ is {\em central\/} of degree $2$,
and $f^{(2)}$ denotes the residue class of the second divided derivative of $f$ in $A'$.

In particular, the {\em even\/} Yoneda $\Ext$--algebra 
$\Ext^{\text{\rm even}}_{A^{\ev}}(A,A)=\oplus_{i\gs 0}\Ext^{2i}_{A^{\ev}}(A,A)$
is the fibre product of the ring $A$ over $A'$ with the polynomial ring $A'[s]$ in one variable of degree $2$,
\[
\Ext^{\text{\rm even}}_{A^{\ev}}(A,A) \cong A\times_{A'}A'[s]\,.
\]
\end{theorem}

\begin{proof}
Just note that the algebra displayed is indeed the cohomology algebra of the DG Clifford algebra
$A\langle t,s\rangle$ with relations $t^{2}=f^{(2)}s, ts=st$, and differential 
$\partial(t)=(\partial_{x}f)s, \partial s =0$. Now apply Theorem \ref{thm:ci}.
\end{proof}

\begin{remark}
\label{2torsion}
While we know from graded commutativity of the algebra that the classes $abf^{(2)}$, for $a,b\in\theta$ 
must be $2$--torsion in $A'$, we can do better, in that these classes are already $2$--torsion in $A$.

To see this, observe that $2f^{(2)}=\partial^{2}f/\partial x^{2}$, the ordinary second derivative of $f$, 
and use that $a\partial_{x}f=uf, b\partial_{x}f=vf$ for suitable polynomials $u,v\in K[x]$, because
$a,b\in\theta$. One then obtains
\begin{align*}
ab\frac{\partial^{2}f}{\partial x^{2}}&= a\left(\partial_{x}(b\partial_{x}f)-
\partial_{x} b\cdot\partial_{x}f\right)
&&\text{by the product rule,}\\
&=a\partial_{x} (vf)-a\partial_{x} b\cdot \partial_{x}f
&&\text{using $b\partial_{x}f=vf$,}\\
&=a(\partial_{x} v) f+av\partial_{x}f- a\partial_{x} b\cdot \partial_{x}f
&&\text{by the product rule again,}\\
&=\left(a\partial_{x} v +uv-u\partial_{x} b\right)f
&&\text{using $a\partial_{x}f=uf$,}\\
&\equiv 0\bmod f.
\end{align*}
\end{remark}

\begin{cor}
If $2$ is a non-zero-divisor in $A=K[x]/(f)$, for $f$ a polynomial in $K[x]$ with content ideal of depth at least two, 
then $\Ext^{\bdot}_{A^{\ev}}(A,A)$ is strictly graded commutative.\qed
\end{cor}

\begin{example}[Inspired by \cite{Rob, DIo}]
The foregoing Corollary fails in more than one variable. Consider $A=\ZZ[x,y]/(x^{2}-4xy +y^{2}-1)$ 
as a $\ZZ$--algebra. 
As $f=x^{2}-4xy +y^{2}-1$ is monic in $y$ of degree $2$, the $\ZZ[x]$--module $A$ is free of rank $2$, in particular, 
$2$ is a non-zero-divisor on $A$ and $A$ is projective over $\ZZ$.

The Euler identity yields 
$x\tfrac{\partial f}{\partial x} + y\tfrac{\partial f}{\partial y}=2f+2\equiv 2$ in $A$, whence the Jacobi ideal of 
$A$ over $\ZZ$ is $(2x - 4y, -4x+2y)= 2A$, and so
$T^{1}_{A/\ZZ} \cong A/(2A)\cong \ZZ_{2}[x,y]/((x+y+1)^{2})$ is a nonzero $2$--torsion $A$--module.

For $D=
\frac{1}{2}\left(\frac{\partial f}{\partial y}\frac{\partial}{\partial x} - \frac{\partial f}{\partial x}\frac{\partial}{\partial y}\right) 
= (2x-y)\frac{\partial}{\partial x}+(x-2y)\frac{\partial}{\partial y}$ one verifies that $D(f)=0$,
already in $\ZZ[x,y]$. Using Theorem \ref{ThmA}, one finds 
\begin{align*}
q(D)(f) &= (2x-y)^{2}\frac{\partial^{(2)}f}{\partial x^{2}} + (2x-y)(x-2y)\frac{\partial^{2}f}{\partial x\partial y}
+(x-2y)^{2}\frac{\partial^{(2)}f}{\partial y^{2}}\\
&= (2x-y)^{2}-4 (2x-y)(x-2y)+(x-2y)^{2}\\
&\equiv y^{2} + x^{2} \equiv 1 \bmod(f,2) \,,
\end{align*}
whence $q(D)(f)$ is not $2$--torsion in the normal module $N_{A/P}\cong A$.

Further, $D\cup D\equiv 1\pmod {2A}$ does not vanish in $T^{1}_{A/K}$, in fact, the squaring map
sends $\Ext^{1}_{A^{\ev}}(A,A)\cong \Der_{\ZZ}(A,A)$ surjectively onto $T^{1}_{A/K}$,
and one may verify easily that further $T^{1}_{A/K}\cong \Ext^{2}_{A^{\ev}}(A,A)\cong \Hoch^{2}(A/K,A)$.
\end{example}

We finish this section by looking at some classical situations.
\subsection*{The Separable or Unramified Case} Recall that a polynomial $f(x)\in K[x]$ 
is {\em separable\/} over $K$, if its derivative $\partial_{x}f$ is a unit modulo $f$, equivalently,
$(f,\partial_{x}f)=K[x]$. In that case, the content ideal $\fc_{f}$ necessarily equals $K$, 
whence $f$ is a non-zero-divisor, $A=K[x]/(f)$ is {\em flat\/} over $K$, and Theorem \ref{mult} applies.
We retrieve that all higher extension groups vanish,  $\Ext^{>0}_{A^{\ev}}(A,A)=0$, as it should be.

\subsection*{The Generically Unramified Case} The polynomial $f$ is 
{\em generically unramified}, if it is a non-zero-divisor in $K[x]$ and its derivative 
$\partial_{x}f$ is a non-zero-divisor modulo $f$. Assuming furthermore that $A$ 
is flat over $K$, equivalently, that the content of $f$ is the unit ideal, Theorem \ref{mult} applies and
shows that $f$ is generically unramified if, and only if, $\Ext^{1}_{A^{\ev}}(A,A)=0$, equivalently,
all odd extension groups $\Ext^{2i+1}_{A^{\ev}}(A,A)$ vanish. In that case, the Yoneda 
$\Ext$--algebra is thus reduced to its even part, the fibre product of $A$ over $A'$
with a polynomial ring $A'[z]$ generated in cohomological degree $2$. 

There is no room for a non-zero squaring map and so the $\Ext$--algebra is strictly graded commutative. 

\subsection*{The Totally Ramified Case} The polynomial $f$ is 
{\em totally ramified}, if it is a non-zero-divisor in $K[x]$, but its derivative 
$\partial_{x}f$ is zero modulo $f$. Equivalently, $A'=A$. 

Assuming that the content of $f$ is the unit ideal, Theorem \ref{mult} shows that
the Yoneda algebra $\Ext^{\bdot}_{A^{\ev}}(A,A)$ is the Clifford algebra over the
$A$--quadratic map $q:At\to As, q(t)= f^{(2)}(x)s$. If $f^{(2)}(x)=0$ in $A$, as happens, for example, 
by Remark \ref{2torsion} if $2$ is a non-zero-divisor in $A$, then the Yoneda algebra is the tensor 
product of the exterior algebra over $A$ generated by $t$ in degree $1$ with the symmetric algebra 
over $A$ generated by $s$ in degree $2$.

In contrast, if $f^{(2)}(x)$ is a unit in $A$, whence $A$ is necessarily of characteristic $2$ 
by Remark \ref{2torsion}, then the Yoneda algebra is the tensor algebra over $A$ on a single 
generator $t$ in degree $1$. This covers the case $A=K[x]/(x^{2})$ with $\car K=2$ mentioned in the Introduction.

\subsection*{Polynomials over a field} Last, but not least, let us consider the case when the
coefficient ring $K$ is a field. Then being a non-zero-divisor simply means that $f$ is not 
the zero polynomial. Trivially, $A$ is projective over $K$, thus the Yoneda $\Ext$--algebra
of self-extensions agrees with Hochschild cohomology.

Further, $K[x]$ being a principal ideal domain, we obtain explicit descriptions of $A'$ and 
$\theta$. Set $g=\gcd(f, \partial_{x}f)$ and $h=f/g$ in $K[x]$, so that $A'\cong K[x]/(g)$
and $\theta = (f(x):\partial_{x}f)/(f) \cong (h)/(f)$. Thus, $\theta$ is the principal ideal generated 
by $h$ in $K[x]/(f)$, isomorphic to $A'$ as $A$--module. Therefore, we obtain the following 
presentation 
\begin{align*}
\Hoch^{\bdot}(A/K)\cong \frac{K[x,y,z]}{\big(f, gy, gz, y^{2}+f^{(2)}h^{2}z\big)}
\end{align*}
of the Hochschild cohomology ring of $A=K[x]/(f)$ over $K$ as a quotient of a graded polynomial ring, 
where $x$ is of degree $0$, $y$ is the class of $h\partial_{x}$ of degree $1$, and $z$ is of degree $2$.

If $\car K\neq 2$, then necessarily $f^{(2)}h^{2}\equiv 0\pmod f$ by Remark \ref{2torsion}
and the description simplifies accordingly.

\begin{remark}
This last example encompasses Theorem 3.2, Lemma 4.1, Lemma 5.1, Theorem 5.2 and Theorem 6.2 from 
\cite{Holm}, completes the characteristic $2$ cases that were not handled there, and covers as well 
Theorem 3.9 from \cite{SuarezAlvarez}.
\end{remark}

\section{Cohomology of Finite Abelian Groups}
\begin{sit}
As our final example, we consider the Hochschild and ordinary cohomology of finite abelian groups.
If $G=\mu_{n_{1}}\times\cdots\mu_{n_{r}}$ is such a group, with $\mu_{n}$ the multiplicatively written
cyclic group of order $n$, we may assume that the orders of the factors satisfy
$2\leqslant n_{1}\mid n_{2}\mid\cdots\mid n_{r}$, so that the $n_{i}$ are the {\em elementary divisors\/} or 
{\em invariant factors\/} of the group.
\end{sit}

\begin{sit}
The group algebra on $G$ over a commutative ring $K$ is then given by
$A=KG \cong K[x_{1},..., x_{r}]/(x_{1}^{n_{1}}-1,..., x_{r}^{n_{r}}-1)$. It is a homological complete intersection over
$K$, thus, Theorem \ref{thm:ci} applies. However, because the variables $x_{i}$ become units in $A$,
one can simplify that result in this case, and recover as well the fact established in \cite{CiSo}
that there is an isomorphism of graded commutative algebras 
$\Hoch^{\bdot}(A/K)\cong A\otimes_{K}\Ext^{\bdot}_{A}(K,K)$, where $K$ is considered 
an $A$--module via the trivial representation of $G$.
\end{sit}

\begin{theorem}
Let $G=\mu_{n_{1}}\times\cdots\mu_{n_{r}}$ be a finite abelian group as above. The Hochschild cohomology
of $KG$ over $K$ is then the cohomology of the tensor product of differential graded Clifford algebras,
\begin{align*}
(C(G),\partial) &=KG\otimes_{K}\bigotimes_{j=1}^{r}
\left(\frac{K\langle \tau_{j}, \sigma_{j}\rangle}{(\tau_{j}^{2}-m_{j}\sigma_{j}, \tau_{j}\sigma_{j}-\sigma_{j}\tau_{j})}, 
\partial \tau_{j} = n_{j}\sigma_{j}, \partial \sigma_{j}=0
\right)\,,
\end{align*}
where
\begin{align*} 
m_{j}&=
\begin{cases}
0&\text{if $n_{j}$ is {\em odd},}\\
n_{j}/2&\text{if $n_{j}$ is {\em even}}\,,
\end{cases}
\end{align*}
and the (cohomological) degree of the $\tau_{j}$ is $1$, that of the $\sigma_{j}$ is $2$.

Further, the cohomology of the second factor in the tensor product is the group cohomology $H^{\bdot}(G,K)=
\Ext^{\bdot}_{KG}(K,K)$, where $K$ is the trivial $G$--representation, and one retrieves the 
isomorphism from \cite{CiSo} of graded commutative algebras
\begin{align*}
\Hoch^{\bdot}(KG/K)\cong KG\otimes_{K}H^{\bdot}(G,K)\,.
\end{align*}
\end{theorem}

\begin{proof}
Note that $A=KG\cong K\mu_{n_{1}}\otimes_{K}\cdots\otimes_{K} K\mu_{n_{r}}$ and that accordingly 
the multiplication map $\mu:A^{\ev}\to A$ is the tensor product over $K$ of the multiplication maps 
$\mu_{j}:(K\mu_{j})^{\ev}\to K\mu_{j}$.

It follows that a Tate model for $A$ over $K$ is obtainable as tensor product of the corresponding Tate models of
the cyclic groups involved. Thus, it remains to discuss the Tate model of a single cyclic group, say $\mu_{n}$,
for some natural number $n$. The group algebra $K\mu_{n}\cong K[x]/(x^{n}-1)$ is a cyclic extension of $K$,
free as $K$--module. Therefore, we can apply \ref{cyclictate} with $f=x^{n}-1 \in K[x]$ to obtain the Tate model
$\calt^{(2)}=\calt$ as a $2$--periodic complex, augmented versus $K\mu_{n}$ in degree $0$, 
\begin{align*}
\xymatrix{
0&\ar[l]K\mu_{n}&\ar[l]0&\ar[l]0&\ar[l]0&\ar[l]\cdots&\equiv\quad K\mu_{n}\\
0\ar[u]&\ar[u]^{\mu}\ar[l]{K\mu_{n}}^{\ev}&\ar[u]\ar[l]_-{x''-x'}{K\mu_{n}}^{\ev}&\ar[u]\ar[l]_-{\Delta}{K\mu_{n}}^{\ev}&
\ar[u]\ar[l]_-{x''-x'}{K\mu_{n}}^{\ev}&\ar[l]\cdots&\equiv\quad\calt^{(2)}\ar[u]_{\epsilon}
}
\end{align*}
where $\Delta\in {K\mu_{n}}^{\ev}$ is the residue class of the unique polynomial
\[
\Delta=\Delta(x',x'')=\frac{(x'')^{n}-(x')^{n}}{x''-x'}=\sum_{i=0}^{n-1}(x'')^{i}(x')^{n-1-i}\in K[x',x'']
\] 
that satisfies $f(x'')-f(x')=(x''-x')\Delta(x',x'')$ in $K[x',x'']$. Written as DG algebra,
\begin{align*}
\calt = {K\mu_{n}}^{\ev}\langle \diff x, \diff f\rangle \cong{\bigwedge}_{{K\mu_{n}}^{\ev}}(\diff x)
\otimes_{{K\mu_{n}}^{\ev}}
\Gamma_{{K\mu_{n}}^{\ev}}(\diff f)\,,
\end{align*}
with $\diff x$ in (homological) degree $1$ and $\diff f$ in degree $2$ and differential
$\partial(\diff f)=\Delta(x',x'')\diff x$ and $\partial(\diff x)=x''-x'$.

Further, the co-unital diagonal approximation $\Phi:\calt\to \calt\otimes_{K\mu_{n}}\calt$ from Proposition
\ref{deg2} is determined by
\begin{align*}
\Phi_{1}(\diff x)&= \diff x' + \diff x''\\
\Phi_{2}(\diff f)&=  \diff x' + \diff x'' -\sum_{p+q+r=n-2}(x')^{p}(\diff x') x^{q} (\diff x'') (x'')^{r}
\end{align*}
in our earlier notation. Next we adapt this co-unital diagonal approximation using 
Theorem \ref{diagadapted}. Setting $x=x'=x''$, we have
\begin{align*}
\Phi_{2}(\diff f) \equiv \diff x' + \diff x'' -\binom{n}{2}x^{n-2}\diff x'\diff x''\bmod (x''-x,x-x') \,.
\end{align*}
Now the form $\diff x'\diff f''\in \calt\otimes_{K\mu_{n}}\calt$ of degree $3$ is in the kernel of both 
augmentation maps $\epsilon_{1,2}$ and satisfies
\begin{align*}
\partial(\diff x'\diff f'') &= (x-x')\diff f'' -\diff x'\Delta(x,x'')\diff x''\\
&\equiv n x^{n-1}\diff x'\diff x''\bmod (x''-x,x-x')\,.
\end{align*}
As $x$ is a unit in $K\mu_{n}$, we can adjust $\Phi_{2}$ by adding $\partial\eta$, where
\begin{align*}
\eta = 
\begin{cases}
\frac{n-1}{2}(x')^{-1}\diff x'\diff f''&\text{if $n$ is {\em odd},}\\
\frac{n}{2}(x')^{-1}\diff x'\diff f''&\text{if $n$ is {\em even},}
\end{cases}
\end{align*}
to obtain the adapted co-unital diagonal approximation $\Psi$ that agrees in degrees $0,1$ with $\Phi$,
but in degree $2$ is given by
\begin{align*}
\Psi(\diff f)&=\Phi_{2}(\diff f) +\partial\eta \equiv \diff x' + \diff x'' \bmod (x''-x,x-x')
\intertext{if $n$ is odd,}
\Psi(\diff f)&=\Phi_{2}(\diff f) +\partial\eta \equiv \diff x' + \diff x''
-\frac{n}{2}x^{n-2}\diff x'\diff x''\bmod (x''-x,x-x')
\end{align*}
if $n$ is even. (This adaptation of the multiplicative structure in case of a cyclic group already occurs in
\cite[Chap. XII \S 7, p.252]{CE}.)

At this stage, $\Hom_{{K\mu_{n}}^{\ev}}(\calt, K\mu_{n})$ with the cup product $\cup_{\Psi}$ and original differential
is the DG Clifford algebra
\begin{align*}
\Hom_{{K\mu_{n}}^{\ev}}(\calt, K\mu_{n}) &\cong \frac{K\mu_{n}\langle t,s\rangle}{(t^{2}- mx^{n-2}s, st-ts)}\,,
\end{align*}
where $m=0$ for $n$ odd, $m=n/2$ for $n$ even, while the differential satisfies $\partial(t) = nx^{n-1}s$ and
$\partial s=0$.

For the final touch, set $\tau= xt, \sigma = x^{n}s$ and note that this coordinate change is indeed invertible, as
$x$ is a unit in $K\mu_{n}$. In terms of $\tau,\sigma$ we obtain
\begin{align*}
\left(\Hom_{{K\mu_{n}}^{\ev}}(\calt, K\mu_{n}),\partial\right) &\cong 
\left(\frac{K\mu_{n}\langle \tau,\sigma\rangle}
{(\tau^{2}- m\sigma, \sigma\tau-\tau\sigma)}, \partial(\tau)= n\sigma, \partial\sigma=0\right)\,.
\end{align*}

Finally, taking the tensor product over $K$ of the Tate models for the cyclic groups involved,
their adapted diagonal approximations fit together to endow the tensor product with a co-unital diagonal 
approximation. Dualizing into $KG$ over $KG^{\ev}$ results in the DG Clifford algebra $(C(G),\partial)$ 
in the statement of the theorem.

To establish the relation to group cohomology, note first that 
\[
H(C(G),\partial) \cong KG\otimes_{K}H\left(\bigotimes_{j=1}^{r}
\left(\frac{K\langle \tau_{j}, \sigma_{j}\rangle}{(\tau_{j}^{2}-m_{j}\sigma_{j}, \tau_{j}\sigma_{j}-\sigma_{j}\tau_{j})}, 
\partial \tau_{j} = n_{j}\sigma_{j}, \partial \sigma_{j}=0
\right)\right)
\]
as $KG$ is flat over $K$. Secondly, observe that applying $K\otimes_{KG}(-)$ to the Tate model $\calt$
of $\mu:KG^{\ev}\to KG$ and then taking cohomology of $\Hom_{KG}(K\otimes_{KG}\calt, K\otimes_{KG}\calt)$
yields a homomorphism of graded algebras
\[
\Ext^{\bdot}_{KG^{\ev}}(KG,KG)\xto{\ K\otimes_{KG}(-)\ } \Ext^{\bdot}_{KG}(K,K)
\]
that factors through $K\otimes_{KG}\Ext^{\bdot}_{KG^{\ev}}(KG,KG)$. Thus, one obtains a 
homomorphism of graded algebras
\[
\label{compare}
\tag{$\diamond$}
H\left(\bigotimes_{j=1}^{r}
\left(\frac{K\langle \tau_{j}, \sigma_{j}\rangle}{(\tau_{j}^{2}-m_{j}\sigma_{j}, \tau_{j}\sigma_{j}-\sigma_{j}\tau_{j})}, 
\partial_{j}
\right)\right)
\lto \Ext^{\bdot}_{KG}(K,K)\,,
\]
where $\partial_{j}\tau_{j} = n_{j}\sigma_{j}, \partial_{j} \sigma_{j}=0$.

One may check directly that this is an isomorphism by noting that one may construct a Tate model for the 
augmentation $KG\to K$ that returns the argument of $H$ in the source of this homomorphism 
and that this is, at least, a $K$--linear isomorphism, if one initially ignores the multiplicative structure.

Alternatively, one may use \cite{SWi} to observe that the structure map
$K\to KG$ induces an algebra homomorphism 
\[
\Ext_{KG}(K,K)\lto KG\otimes \Ext_{KG}(K,K)\cong \Hoch(KG/K)
\]
that provides a splitting to the $K$--algebra 
homomorphism (\ref{compare}).
\end{proof}

\begin{remark}
One obstacle to an even more transparent proof that (\ref{compare}) is an isomorphism is that
the co-unital diagonal approximation on the Tate model for $\mu:KG^{\ev}\to KG$ does not obviously induce 
such an approximation when applying $K\otimes_{KG}(-)$. This obstacle can be overcome by exhibiting
directly such a co-unital diagonal approximation on the Tate model for the augmentation $KG\to K$, as is done 
in \cite{Roberts}, so that it will agree with the one on $K\otimes_{KG}C(G)$ exhibited above.

An explicit co-unital diagonal approximation on the Tate model for the augmentation $KG\to K$ has also been 
described in \cite{Hue} and we refer to that reference and \cite{Roberts} for explicit information on the structure of
$H(G,K)$, augmenting and reproducing results from \cite{Chap, Town}.
\end{remark}



\begin{thebibliography}{bmyt}
\bibitem{And}
{Andr\'e, M.:}
\textit{Homologie des alg\`ebres commutatives}.
Die Grundlehren der mathematischen Wissenschaften, Band \textbf{206}. 
Springer-Verlag, Berlin-New York, 1974. xv+341 pp.

\bibitem{Avr1}
{Avramov, L.~L.:}
\textit{Local algebra and rational homotopy}.
Algebraic homotopy and local algebra (Luminy, 1982), 15--43, 
Ast\'erisque, 113-114, Soc. Math. France, Paris, 1984.


\bibitem{Avr2}
{Avramov, L.~L.:}
\textit{Infinite free resolutions}. Six lectures on commutative algebra (Bellaterra, 1996), 1--118, 
Progr. Math., \textbf{166}, Birkh¬\"auser, Basel, 1998.

\bibitem{ABS}
{Avramov, L.~L.; Bonacho Dos Anjos Henriques,  I.~B.; \c{S}ega, L.~M.:}
\textit{Quasi-complete intersection homomorphisms}.
preprint 2010, 20 pages, \url{arXiv:1010.2143}

\bibitem{AvVi}
{Avramov, L.~L.; Vigu\'e-Poirrier, M.:}
\textit{Hochschild homology criteria for smoothness}. 
Internat. Math. Res. Notices \textbf{1992}, no. 1, 17--25. 

\bibitem{BW}
{Bergeron, N.; Wolfgang, H.L.:}
\textit{The decomposition of Hochschild cohomology and Gerstenhaber operations}. 
J.~Pure Appl.~Algebra \textbf{104} (1995), no. 3, 243--265. 

\bibitem{BMR}
{Blanco, A.; Majadas, J.; Rodicio, A.~G.:}
\textit{On the acyclicity of the Tate complex}. 
J.~Pure Appl.~Algebra \textbf{131} (1998), no.~2, 125--132.

\bibitem{BS}
{Bonacho Dos Anjos Henriques, I.; \c{S}ega, L.~M.:}
\textit{Free resolutions over short Gorenstein local rings}. 
Math.~Z.~\textbf{267} (2011), no.~3-4, 645--663.

\bibitem{BM}
{Brewer, J.~W.; Montgomery, P.~R.:}
\textit{The finiteness of $I$ when $R[X]/I$ is $R$--projective}. 
Proc. Amer. Math. Soc. \textbf{41} (1973), 407--414.

\bibitem{BKu}
{Br\"uderle, S.; Kunz, E.:} 
\textit{Divided powers and Hochschild homology of complete intersections}. 
With an appendix by Reinhold H\"ubl. 
Math.~Ann. \textbf{299} (1994), no. 1, 57--76. 

\bibitem{BFl} {Buchweitz, R.-O.; Flenner, H.:}
\textit{Global Hochschild (co-)homology of singular spaces}.
Adv.~Math. \textbf{217} (2008), no.~1, 205--242.

\bibitem{BACH}
{Buenos Aires Cyclic Homology Group:}
\textit{Hochschild and cyclic homology of hypersurfaces}. 
Adv.~Math.~\textbf{95} (1992), no.~1, 18--60.

\bibitem{CvdB}
{Calaque, D.; Van den Bergh, M.:} 
\textit{Hochschild cohomology and Atiyah classes}. 
Adv. Math. \textbf{224} (2010), no.~5, 1839--1889. 

\bibitem{SemCartan}
{Cartan, H.:}
\textit{Alg\`ebre d'Eilenberg-Mac Lane et Homotopie.} 
S\'eminaire Henri Cartan \textbf{7} (1954-55);
available in electronic form through NumDam.

\bibitem{CE} 
{Cartan, H. and Eilenberg, S.:} 
\textit{Homological algebra},
Princeton University Press, 1956.

\bibitem{Chap}
{Chapman, G.~R.:} 
\textit{The cohomology ring of a finite abelian group}. 
Proc.~London Math.~Soc.~(3) \textbf{45} (1982), no.~3, 564--576.

\bibitem{CiSo} 
 {Cibils, C.~and Solotar, A.:} 
 \textit{Hochschild cohomology algebra of abelian groups}. 
 Arch. Math. (Basel) \textbf{68} (1997), no. 1, 17--21. 
 
\bibitem{CuQui}
{Cuntz, J.~and Quillen, D.:}
\textit{Algebra extensions and nonsingularity}.
J.~Amer.~Math.~Soc. \textbf{8}, no.~2 (1995), 251--289.

\bibitem{DIo}
{Dumitrescu, T.; Ionescu, C.:}
\textit{Some examples of two-dimensional regular rings}.
preprint 2013, 6 pages, \url{arXiv:1301.2615 }

\bibitem{Ger1}
{Gerstenhaber, M.:}
\textit{The cohomology structure of an associative ring}.
Ann.~of Math.~(2) \textbf{78} (1963), 267--288

\bibitem{EGA}
{Grothendieck, A.:}
\textit{El\'ements de g\'eom\'etrie alg\'ebrique. IV. \'Etude locale des sch\'emas 
et des morphismes de sch\'emas. I}. 
Inst. Hautes \'Etudes Sci. Publ. Math. No. \textbf{20} (1964), 259 pp. 

\bibitem{GG}
{Guccione, J.~A.; Guccione, J.~J.:}
\textit{Hochschild homology of complete intersections}. 
J.~Pure Appl.~Algebra \textbf{74} (1991), no.~2, 159--176

\bibitem{Holm}
{Holm, T.:}
\textit{Hochschild cohomology rings of algebras $k[X]/(f)$}. 
Beitr\"age Algebra Geom.~\textbf{41} (2000), no. 1, 291--301.

\bibitem{Hue}
{H\"ubschmann, J.:}
\textit{Cohomology of finitely generated abelian groups}. 
Enseign.~Math.~(2) \textbf{37} (1991), no. 1-2, 61--71.

\bibitem{Iy}
{Iyengar, S.:} 
\textit{Acyclicity of Tate constructions}. 
J.~Pure Appl.~Algebra \textbf{163} (2001), no. 3, 289--300.

\bibitem{MRo}
{Majadas, J.; Rodicio, A.~G.:} 
\textit{The Hochschild (co-)homology of hypersurfaces}. 
Comm. Algebra \textbf{20} (1992), no. 2, 349--386.

\bibitem{Nor}
{Northcott, D.~G.:}
\textit{Finite free resolutions}. 
Cambridge Tracts in Mathematics, No. \textbf{71}. 
Cambridge University Press, Cambridge-New York-Melbourne, 1976. xii+271 pp.

\bibitem{OR}
{Ohm, J.; Rush, D.~E.:} 
\textit{The finiteness of $I$ when $R[X]/I$ is flat}. 
Trans.~Amer.~Math.~Soc. \textbf{171} (1972), 377--408.

\bibitem{Qui} 
{Quillen, D.:} 
\textit{On the (co-)homology of commutative rings.}  
Applications of Categorical Algebra (Proc.\ Sympos.\ Pure Math., Vol.\ {\bf XVII},
New York, 1968) pp.\ 65--87. Amer.\ Math.\ Soc., Providence, R.I., 1970.

\bibitem{Roberts}
{Roberts, C.~D.:}
\textit{The Cohomology Ring of a Finite Abelian Group}.
Thesis (Ph.D.)-- University of Waterloo. 2013. 141 pages. 

\bibitem{Rob}
{Roberts, L.~G.:}
\textit{Some examples of smooth and regular rings}. 
Canad.~Math.~Bull.~\textbf{23} (1980), no.~3, 255--259. 

\bibitem{San}
{Sanada, K.:} 
\textit{On the Hochschild cohomology of crossed products}. 
Comm. Algebra \textbf{21} (1993), no. 8, 2727--2748.

\bibitem{SWi}
{Siegel, S.~F.; Witherspoon, S.~J.:}
\textit{The Hochschild cohomology ring of a group algebra}. 
Proc.~London Math.~Soc.~(3) \textbf{79} (1999), no. 1, 131--157. 

\bibitem{Sjo}
{Sj¬\"odin, G.:} 
\textit{Hopf algebras and derivations}. 
J.~Algebra \textbf{64} (1980), no.~1, 218--229.
 
\bibitem{SA}
{Su\'arez-Alvarez, M.:} 
\textit{The Hilton-Heckmann argument for the anti-commutativity of cup products}. 
Proc. Amer. Math. Soc. \textbf{132} (2004), no. 8, 2241--2246

\bibitem{SuarezAlvarez}
{Su\'arez-Alvarez, M.:}
\textit{Applications of the change-of-rings spectral sequence to the computation of Hochschild cohomology}.
preprint 2007, 27 pages, \url{arXiv:0707.3210}

\bibitem{Ta} 
{Tate, J.:}
\textit{Homology of Noetherian rings and local rings}. 
Illinois J.~Math.~\textbf{1} (1957), 14--27.

\bibitem{Town}
{Townsley Kulich, L.~G.:}
\textit{Investigations of the integral cohomology ring of a finite group}. 
Thesis (Ph.D.)--Northwestern University. 1988. 83 pp. 

\bibitem{Wol}
{Wolffhardt, K.:}
\textit{The Hochschild homology of complete intersections}. 
Trans. Amer. Math. Soc.~\textbf{171} (1972), 51--66

\end{thebibliography}
\end{document}